\documentclass[twoside, 11pt]{amsart}
\usepackage[letterpaper,hmargin=1.25in,vmargin=1.5in]{geometry}
\usepackage{amscd}
\usepackage{verbatim}
\usepackage{color}
\usepackage{comment}

\input xy
\xyoption{all}

\usepackage{amssymb}
\usepackage{hyperref}

\DeclareMathAlphabet{\cat}{OT1}{cmss}{m}{sl}

\newtheorem*{theorem*}{Theorem}
\newtheorem{theorem}{Theorem}[section]

\newtheorem{proposition}[theorem]{Proposition}
\newtheorem{lemma}[theorem]{Lemma}
\newtheorem{corollary}[theorem]{Corollary}

\theoremstyle{definition}
\newtheorem{remark}[theorem]{Remark}

\newcommand{\tens}{\otimes}

\newcommand{\gmu}{\boldsymbol{\mu}}

\newcommand{\id}{\mathrm{id}}

\newcommand{\CH}{\operatorname{CH}}

\newcommand{\C}{\operatorname{\cat{C}}}

\renewcommand{\Im}{\operatorname{Im}}

\newcommand{\Ker}{\operatorname{Ker}}

\newcommand{\ind}{\operatorname{\hspace{0.3mm}ind}}

\newcommand{\Inv}{\operatorname{Inv}}
\newcommand{\res}{\operatorname{res}}

\newcommand{\disc}{\operatorname{disc}}

\newcommand{\Br}{\operatorname{Br}}

\newcommand{\SB}{\operatorname{SB}}
\newcommand{\SBS}{\operatorname{\cat{SB}}}

\newcommand{\gSL}{\operatorname{\mathbf{SL}}}

\newcommand{\tors}{\operatorname{\cat{tors}}}
\newcommand{\rank}{\operatorname{rank}}
\newcommand{\PQ}{\operatorname{\cat{PQ}}}
\newcommand{\QS}{\operatorname{\cat{Q}}}

\renewcommand{\P}{\mathbb{P}}
\newcommand{\Z}{\mathbb{Z}}

\newcommand{\Q}{\mathbb{Q}}

\newcommand{\cA}{\mathcal A}

\newcommand{\cF}{\mathcal F}

\newcommand{\cM}{\mathcal M}

\title[Codimension 2 cycles on products of  homogeneous surfaces] 
{Codimension 2 cycles on products of projective homogeneous surfaces}

\author
[S. Baek] {Sanghoon Baek}

\address{Department of Mathematical Sciences, KAIST, 291 Daehak-ro, Yuseong-gu, Daejeon, 305-701, Republic of Korea}

\email {sanghoonbaek@kaist.ac.kr}

\begin{document}

\maketitle

\begin{abstract}
In the present paper, we provide general bounds for the torsion in the codimension $2$ Chow groups of the products of projective homogeneous surfaces. In particular, we determine the torsion for the product of four Pfister quadric surfaces and the maximal torsion for the product of three Severi-Brauer surfaces. We also find an upper bound for the torsion of the product of three quadric surfaces with the same discriminant.
\end{abstract}

\tableofcontents

\section{Introduction}\label{intro}

Let $X$ be a smooth projective homogeneous variety under the action of a semisimple group over an algebraically closed field $F$. The Chow group $\CH(X)$ of algebraic cycles modulo the rational equivalence is well-understood as well as its ring structure. Namely, the Chow group of $X$ is a free abelian with the basis of Schubert cycles. For an arbitrary base field $F$, this is no longer true: the Chow group $\CH(X)$ can have torsion. Indeed, by a transfer argument the problem of determining $\CH(X)$ over an arbitrary field $F$ reduces to computing its torsion subgroup $\CH(X)_{\tors}$. 

For codimension $d\leq 1$, the Chow group $\CH^{d}(X)$ is torsion-free. A nontrivial torsion first appears when $d=2$ and the exact structure of $\CH^{2}(X)_{\tors}$ is known in many cases. For a projective quadric $X$, the group $\CH^{2}(X)_{\tors}$ is either $0$ or $\Z/2\Z$ \cite{Kar91}. For a Severi-Brauer variety $X$, it is shown that the group $\CH^{2}(X)_{\tors}$ is either $0$ or a cyclic group if the corresponding algebra satisfies certain conditions \cite{Kar}. For a simple simply connected group which splits over $F(X)$, it is known that the group $\CH^{2}(X)_{\tors}$ is a cyclic group generated by the Rost invariant \cite{GMS}, \cite{Pey95}. However, the only partial results are known for their products. In this case, the structure of the torsion subgroup is more complicated.

We consider in this paper the product of two dimensional flag varieties of the same type. For an integer $n\geq 1$, this can be divided into two classes: the set of all products of $n$ Severi-Brauer surfaces, denoted by $\SBS_{n}$ and the set of all products of $n$ quadric surfaces, denoted by $\QS_{n}$. The latter has a special subclass $\PQ_{n}$ consisting of all products of $n$ Pfister quadric surfaces. Here, we view the product of two conics as the product of two Pfister quadric surfaces as they have the same torsion subgroup  in $\CH^{2}$ by \cite[Corollary 2.5]{IzhKar98}. To measure the order of the torsion subgroup of codimension $2$ cycles, we introduce the following notation: $\cM(\cA)=\max_{X\in \cA}\{ |\CH^{2}(X)_{\tors}|\}$, where $\cA=\SBS_{n}$ or $\QS_{n}$ or $\PQ_{n}$. Hence, the group $\CH^{2}(X)_{\tors}$ for any $X\in \cA$ is an elementary abelian group whose order is a divisor of $\cM(\cA)$. We denote by $\C_{n}$ the set of all products of $n$ conics, thus $\cM(\PQ_{n})=\cM(\C_{n})$.

It is well-known that $\cM(\SBS_{1})=\cM(\QS_{1})=\cM(\C_{2})=1$. In \cite{Pey} Peyre proved that $\cM(\C_{3})=2$. In \cite{IzhKar}, \cite{IzhKar98} and \cite{IzhKar2000} Izhboldin and Karpenko proved that $\cM(\SBS_{2})=3$ and $\cM(\QS_{2})=2$. As it is showed in \cite[Theorem 2.1]{Pey95} and \cite[Theorem 4.1]{Pey}, the torsion subgroup of codimension $2$ cycles is closely related to a relative Galois cohomology group in degree $3$ and the above results were used to describe the cohomology groups. 

The main goal of this paper is to generalize their results to arbitrary $n$. More precisely, our main result says that
\begin{theorem*}
For any $n\geq 2$, we have $\cM(\C_{n})\!=\!\cM(\PQ_{n})\!=\!2^{N}$ and $\cM(\SBS_{n})\geq 3^{M}$, where
\[N=2^{n}-({{n}\choose{2}}+n+1)\, \text{ and }\,  M=2^{n}+4{{n}\choose{3}}-(n+1).\]
Moreover, $(1)$ the group $\CH^{2}(X)_{\tors}$ admits all elementary abelian group whose order is a divisor of $2^5$ for $X\in \C_{4}$. $(2)$ $\cM(\SBS_{3})=3^{8}$.
\end{theorem*}

Our results $\cM(\PQ_{4})=2^{5}$ and $\cM(\SBS_{3})=3^{8}$ give the first examples such that the group $\CH^{2}(X)_{\tors}$ is not cyclic in their classes. Moreover, we determine the torsion of $\C_{4}$ in terms of the indexes of the corresponding algebras (Theorem \ref{fourconicsthm}) and present its application to a Galois cohomology in degree $3$. Another direct application of our results (Proposition \ref{threeconics}, Theorem \ref{fourconicsthm}) is given in \cite{Baek}. These results provide the key ingredients to determine the degree $3$ invariants of the quotient of $(\gSL_{2})^{n}$ by its maximal central subgroup (see \S\ref{applicationgalcoh} for details). This also indicates possible applications of our results on the product of Severi Brauer surfaces. In the last part, we provide an elementary short proof of Izhboldin and Karpenko's result for $\cM(\QS_{2})$ (Theorem \ref{twoquadricsurface}, Proposition \ref{lowerboundgammaintwoquadric}) that does not use any cohomological method or $K$ theory of quadrics, and find an upper bound for the torsion of the product of three quadric surfaces with the same discriminant (Proposition \ref{threequadricsurface}).


As shown in \cite{Kar95}, \cite{Kar}, the topological and gamma filtrations on the Grothendieck ring can be used to find the torsion in $\CH^{2}$ of projective homogeneous varieties. Moreover, the torsion in the filtrations can be computed by studying the divisibility of certain polynomials produced by the Chern classes on the Grothendieck ring. We extend this approach to obtain the torsion in $\CH^{2}$ on the product of projective homogeneous surfaces by using new combinatorial results.

This paper is organized as follows. In Section \ref{section2}, we recall basic definitions and facts used in the rest of the paper. In Section \ref{section4}, we determine $\cM(\C_{n})$ and the torsion of $\C_{4}$. In the last part of the section, we present its application to a Galois cohomology group. In Section \ref{section5}, we find a general lower bound for $\cM(\SBS_{n})$. Using this bound, we determine $\cM(\SBS_{3})$. In the last section, we recover a result of Izhboldin and Karpenko and extend it to the product of three quadric surfaces with the same discriminant. 

In the present paper, we denote by $A_{\tors}$ the torsion subgroup of an abelian group $A$ and $I_{n}=\{1,\ldots, n\}$ for any integer $n\geq 1$. 


\section{Two filtrations on the Grothendieck ring}\label{section2}
In this section, we briefly recall definitions and properties of the topological filtration and the gamma filtration on the Grothendieck ring $K(X)$ of a smooth projective variety $X$ (see \cite{FL}, \cite{Kar} for details). We also provide a useful fact concerning the torsion part of these two filtrations on $X$. The topological filtration 
\[K(X)=T^{0}(X)\supset T^{1}(X)\supset \cdots \]
is given by the ideal $T^{d}(X)$ generated by the class $[\mathcal{O}_{Y}]$ of the structure sheaf of a closed subvariety $Y$ of codimension at least $d$. We write $T^{d/d+1}(X)$ for the quotient $T^{d}(X)/T^{d+1}(X)$.

Let $\Gamma^{0}(X)=K(X)$ and $\Gamma^{1}(X)=\Ker(\rank:K(X)\to \Z)$. The gamma filtration 
\[K(X)=\Gamma^{0}(X)\supset \Gamma^{1}(X)\supset \cdots \]
is given by the ideals $\Gamma^{d}(X)$ generated by $\gamma_{d_{1}}(x_{1})\cdots \gamma_{d_{i}}(x_{i})$ with $x_{i}\in \Gamma^{1}(X)$ and $d_{1}+\cdots +d_{i}\geq d$, where $\gamma_{d_{i}}$ is the gamma operation on $K(X)$. For instance, we get $\gamma_{0}(x)=1$ and $\gamma_{1}=\id$, where $x\in K(X)$. Indeed, the gamma operation defines the Chern class $c_{i}(x):=\gamma_{i}(x-\rank(x))$ with values in $K$. We write $\Gamma^{d/d+1}(X)$ for the quotient $\Gamma^{d}(X)/\Gamma^{d+1}(X)$.

For any $d\geq 0$, the gamma filtration $\Gamma^{d}(X)$ is contained in the topological filtration $T^{d}(X)$. For small degree $d=1, 2$, two filtrations coincide. Moreover, the second quotient of the topological filtration can be identified with the codimension $2$ cycles so that we have:
\begin{equation}\label{tolchow}
\Gamma^{2/3}(X)\twoheadrightarrow T^{2/3}(X)=\CH^{2}(X).
\end{equation}

Now we assume that $X$ is a smooth projective homogeneous variety over a field $F$. Let $E$ be a splitting field of $X$. Then, by \cite[Proposition 3.4]{Pey95} we have
\begin{equation}\label{gammatwo}
T^{d}(X)=\Gamma^{d}(X)=\Gamma^{d}(X_{E})\cap K(X)\, \text{ for } d=1, 2.
\end{equation}

Let $\cF$ be either the gamma-filtration $\Gamma$ or the topological filtration $T$ on $K(X)$. Applying the Snake lemma to the commutative diagram involving the exact sequences $0\to \cF^{d+1}(X)\to \cF^{d}(X)\to \cF^{d/d+1}(X)\to 0$ and the one over a splitting field $E$ of $X$, we have the following useful formula \cite[Proposition 2]{Kar95}:
\begin{equation}\label{alphalem}
|\!\oplus \cF^{d/d+1}(X)_{\tors}|\cdot|K(X_{E})/K(X)|=\prod_{d=1}^{\dim(X)}|\cF^{d/d+1}(X_{E})/\Im(\res^{d/d+1})|,
\end{equation}
where $\res^{d/d+1}: \cF^{d/d+1}(X)\to \cF^{d/d+1}(X_{E})$ is the restriction map.

\subsection{Grothendieck group of the product of Severi-Brauer varieties}\label{section3} We recall some basic facts about the Grothendieck group $K$ and $\CH^{2}$ of a product of Severi-Brauer varieties.

Let $A_{i}$ be a central simple $F$-algebra of degree $d_{i}$ for $1\leq i\leq n$. Consider the restriction 
\begin{equation}\label{Quillengen}
K(\prod_{i=1}^{n}\SB(A_{i}))\to K(\prod_{i=1}^{n}\P^{d_{i}-1}_{E}),
\end{equation}
where the corresponding Severi-Brauer variety $\SB(A_{i})$ over a splitting field $E$ is identified with the projective space $\P^{d_{i}-1}_{E}$ for all $i$. The latter ring in (\ref{Quillengen}) is isomorphic to the quotient ring $\Z[x_{1},\cdots, x_{n}]/((x_{1}-1)^{d_{1}},\cdots, (x_{n}-1)^{d_{n}})$, where $x_{i}$ is the pullback of the class of the tautological line bundle on $\P^{d_{i}-1}_{E}$. Then by \cite[\S 8 Theorem 4.1]{Qui} (see also \cite[Proposition 3.1]{Pey}) the image of the map (\ref{Quillengen}) coincides with the sublattice with basis 
\begin{equation}\label{Quillenbasis}
\{ \ind(A_{1}^{\tens i_{1}}\tens \cdots \tens A_{n}^{\tens i_{n}})\cdot x_{1}^{i_{1}}\cdots x_{n}^{i_{n}} \,|\,\, 0\leq i_{j}\leq d_{j}-1, 1\leq j\leq n\}.
\end{equation}

Let $X$ be the product of Severi-Brauer varieties $\SB(A_{i})$ as above. For codimension $2$ cycles, one can simplify the computation of torsion subgroup by using \cite[Proposition 4.7]{IzhKar}: if $\langle A'_{1}, \ldots, A'_{n'}\rangle=\langle A_{1},\ldots, A_{n}\rangle$ in the Brauer group $\Br(F)$, then for $X'=\prod_{i=1}^{n'} \SB(A'_{i})$ \begin{equation}\label{reductionSB}
\CH^{2}(X)_{\tors}\simeq \CH^{2}(X')_{\tors}.
\end{equation}

Now we restrict our attention to $p$-primary algebras. Let $A_{1},\ldots, A_{n}$ be central simple algebras of $p$-power degree with given indices $\ind(A_{1}^{\tens i_{1}}\tens \cdots \tens A_{n}^{\tens i_{n}})$ for all nonnegative integers $i_{1},\ldots, i_{n}$. Let $X$ be the product of $\SB(A_{1}),\ldots, \SB(A_{n})$. Then, by \cite[Corollary 2.15]{Kar} the map (\ref{tolchow}) induces a surjection on torsion subgroups $\Gamma^{2/3}(X)_{\tors}\twoheadrightarrow \CH^{2}(X)_{\tors}$. Moreover, by \cite[Theorem 4.5]{IzhKar} and \cite[Proposition III.1]{Kar99} there is a product $\bar{X}$ of Severi-Brauer varieties $\SB(\bar{A}_{i})$ of algebras $\bar{A}_{i}$ with $\ind(A_{1}^{\tens i_{1}}\tens \cdots \tens A_{n}^{\tens i_{n}})=\ind(\bar{A}_{1}^{\tens i_{1}}\tens \cdots \tens \bar{A}_{n}^{\tens i_{n}})$ for all integers $i_{1},\ldots, i_{n}\geq 0$ such that the above surjective map becomes a bijection
\begin{equation}\label{generic}
\Gamma^{2/3}(\bar{X})_{\tors}\overset\sim\to \CH^{2}(\bar{X})_{\tors}.
\end{equation}
This variety $\bar{X}$ will be called a generic variety corresponding to $X$.

\section{Product of Conics}\label{section4}

In the present section, we provide a general bound for the torsion in codimension $2$ cycles of the product of $n$ conics in Corollary \ref{maximaltorsionconicdetermined}. In case of the product of four conics, we determine its torsion in terms of the indexes of the corresponding algebras in Theorem \ref{fourconicsthm}. In the last subsection, we present its application to a Galois cohomology group. 

For the product of two conics $X$, it is well-known that $\CH^{2}(X)_{\tors}=0$ \cite[Corollary 3.9]{Pey}. We determine the Chow group of codimension $2$ of the product of three conics. 
\begin{proposition}$($\text{cf}. \cite[Proposition 6.1, Proposition 6.3]{Pey}$)$\label{threeconics}
Let $Q_{1}, Q_{2}, Q_{3}$ be quaternions and $X=\SB(Q_{1})\times \SB(Q_{2})\times \SB(Q_{3})$. Then, we have $\cM(\C_{3})=\cM(\PQ_{3})=2$. Moreover, $\CH^{2}(X)_{\tors}$ is trivial except the cases where the division algebras $Q_{1}, Q_{2}, Q_{3}$ satisfy relations $\ind(Q_{i}\tens Q_{j})=\ind(Q_{1}\tens Q_{2}\tens Q_{3})=2$ or $4$ for all $1\leq i\neq j\leq 3$ and in these cases $\CH^{2}(\bar{X})_{\tors}=\Z/2\Z$, where $\bar{X}$ is the corresponding generic variety.
\end{proposition}
\begin{proof}
Let $d=\ind(Q_{1}\tens Q_{2}\tens Q_{3})$, $d_{ij}=\ind(Q_{i}\tens Q_{j})$ for $1\leq i\neq j\leq 3$, and $X=\SB(Q_{1})\times \SB(Q_{2})\times \SB(Q_{3})$. If one of $\ind(Q_{i})$, $d_{ij}$, $d$ is equal to $1$, then by (\ref{reductionSB}) the group $\CH^{2}(X)_{\tors}$ is trivial. Hence, by (\ref{Quillenbasis}) we may assume that we have the following basis 
\[\{1, 2y_{i}, d_{ij}y_{i}y_{j}, d(y_{1}y_{2}y_{3}+y_{1}y_{2}+y_{1}y_{3}+y_{2}y_{3})\, | \, 1\leq i\neq j\leq 3\}\]
of $K(X)$, where $y_{i}=x_{i}-1$, $x_{i}$ is the pullback of the tautological line bundle on the projective line over a splitting field $E$, and $d_{ij}, d\geq 2$.

If either $d=8$ or $d=4$ and $d_{ij}=2$ for some $i, j$, then the subgroup $K(X)\cap T^{3}(X_{E})$ is generated by $dy_{1}y_{2}y_{3}$. As $d_{ij}y_{i}y_{j}\in T^{2}(X)$ and $2y_{k}\in T^{1}(X)$ for $k\neq i, j$, we have $dy_{1}y_{2}y_{3}=2y_{k}(d_{ij}y_{i}y_{j})\in T^{3}(X)$, which together with (\ref{gammatwo}) implies that $T^{2/3}(X)_{\tors}=\CH^{2}(X)_{\tors}=0$.

We now turn to the remaining cases. Let $\alpha_{n}=|\Gamma^{n/n+1}(X_{E})/\Im(\res^{n/n+1})|$ for $1\leq n\leq 3$. Then, we obtain $\alpha_{1}\leq 2^{3}$. As $c_{2}(dx_{1}x_{2}x_{3})={{d}\choose{2}}(6y_{1}y_{2}y_{3}+2y_{1}y_{2}+2y_{1}y_{3}+2y_{2}y_{3})\in \Gamma^{2}(X)$, we get $2(y_{1}y_{2}+y_{1}y_{3}+y_{2}y_{3})\in \Im(\res^{2/3})$ if $d=2$. If in addition $d_{ij}=4$ for some $i, j$, then by replacing $y_{i}y_{j}$ in codimension $2$ part of basis of $K(X_{E})$ with $y_{1}y_{2}+y_{1}y_{3}+y_{2}y_{3}$ we have
\[\alpha_{2}\leq
\begin{cases}
2\cdot d_{ik}\cdot d_{jk} & \text{ if } d=2, d_{ij}=4  \text{ for some } 1\leq i\neq j\leq 3,\\
d_{12}d_{13}d_{23} & \text{ otherwise}.\\
\end{cases}\]
Since $c_{2}(dx_{1}x_{2}x_{3})\cdot 2y_{1}=4{{d}\choose{2}}y_{1}y_{2}y_{3}\in \Gamma^{3}(X)$, we get $\alpha_{3}\leq 2d$. As $|K(X_{E})/K(X)|=2^{3}d_{12}d_{13}d_{23}d$, it follows from (\ref{alphalem}) that 
\begin{equation}\label{propositioncase1}
|\!\oplus \Gamma^{n/n+1}(X)_{\tors}|\leq 2\,\,  \text{ if } d=d_{ij}=2 \text{ or } d=d_{ij}=4
\end{equation}
for all $1\leq i\neq j\leq 3$. Otherwise, $|\!\oplus \Gamma^{n/n+1}(X)_{\tors}|=1$, thus, $\CH^{2}(X)_{\tors}=0$.

In the case of (\ref{propositioncase1}), the class of $dy_{1}y_{2}y_{3}$ gives a torsion of order $2$ in $\Gamma^{2/3}(X)$ as $\Gamma^{3}(X)$ is generated by $\Gamma^{1}(X)\cdot \Gamma^{2}(X)$ and any element of $\Gamma^{1}(X)\cdot \Gamma^{2}(X)$ is divisible by $2d$.  Hence, we have $\Gamma^{2/3}(X)_{\tors}=\Z/2\Z$ and $\CH^{2}(\bar{X})_{\tors}=\Z/2\Z$ by (\ref{generic}).\end{proof}

We determine the maximal torsion in Chow group of codimension $2$ of the product of $n$ conics or the product of $n$ Pfister quadric surfaces, which gives a lower bound of $\cM(\QS_{n})$.

\begin{proposition}\label{lowerboundofconics}
Let $n\geq 2$ be an integer and let $Q_{1},\ldots, Q_{n}$ be quaternion algebras satisfying $\ind(Q_{i_{1}}\tens \cdots \tens Q_{i_{m}})=2$ for all $1\leq i_{1}<\cdots <i_{m}\leq n$ and $1\leq m\leq n$. Then, the torsion subgroup $\CH^{2}(\bar{X})_{\tors}$ of the corresponding generic variety of $X=\prod_{i=1}^{n}\SB(Q_{i})$ is 
\[(\Z/2\Z)^{\oplus N}, \text{ where } N=2^{n}-({{n}\choose{2}}+n+1).\]
In particular, $\cM(\C_{n})=\cM(\PQ_{n})\geq 2^{N}$.
\end{proposition}
\begin{proof}
Let $Q_{1},\ldots, Q_{n}$ be quaternion algebras such that $\ind(Q_{i_{1}}\tens \cdots \tens Q_{i_{m}})=2$ for all $1\leq i_{1}<\cdots <i_{m}\leq n$, $1\leq m\leq n$ and $X=\prod_{i=1}^{n}\SB(Q_{i})$. Then, by (\ref{Quillenbasis}) we have a basis $\{1, 2x_{i_{1}}\cdots x_{i_{m}}\,|\, 1\leq m\leq n,\, 1\leq i_{1}<\!\cdots\!<i_{m}\leq n\}$ of $K(X)$, where for each $i\in I_{n}$, $x_{i}$ is the pullback of the class of the tautological line bundle on the projective line over a splitting field $E$. Consider another basis $\{1, 2y_{i_{1}}\cdots y_{i_{m}}\,|\, 1\leq m\leq n, 1\leq i_{1}<\!\cdots\!<i_{m}\leq n\}$, where $y_{i}=x_{i}-1$ for all $i\in I_{n}$. 

Fix $j\geq 1$. As $c_{1}(2x_{i_{1}}\cdots x_{i_{m}})$ and $c_{2}(2x_{i_{1}}\cdots x_{i_{m}})$ are divisible by $2$, any element of $\Gamma^{2j+1}(X)$ is divisible by $2^{j+1}$. Therefore, as $2y_{i_{1}}\cdots y_{i_{m}}\in \Gamma^{2}(X)$ for all $2\leq m\leq n$, we have
\begin{equation}\label{conicpf1}
2^{j}y_{i_{1}}\cdots y_{i_{m}}\in \Gamma^{2j}(X)\backslash \Gamma^{2j+1}(X)
\end{equation}
for any $2j+1\leq m\leq n$. Moreover, by multiplying $2y_{i_{m+1}}\in\Gamma^{1}(X)$ to (\ref{conicpf1}) we get
\begin{equation}\label{conicpf2}
2^{j+1}y_{i_{1}}\cdots y_{i_{m}}\in \Gamma^{2j+1}(X).
\end{equation}
for any $2j+2\leq m\leq n$. Since $c_{2}(2x_{i_{1}}x_{i_{2}})\cdots c_{2}(2x_{i_{2j-1}}x_{i_{2j}})c_{1}(2x_{i_{2j+1}})=2^{j+1}y_{i_{1}}\cdots y_{i_{2j+1}}\in \Gamma^{2j+1}(X)$, it follows from (\ref{conicpf1}) and (\ref{conicpf2}) that for any $j\geq 1$ and any $2j+1\leq m\leq n$ the class of $2^{j}y_{i_{1}}\cdots y_{i_{m}}$ generates a subgroup of $\Gamma^{2j/2j+1}(X)_{\tors}$ of order $2$. By the divisibility of an element of $\Gamma^{2j+1}(X)$, any two subgroups generated by different classes of the elements in (\ref{conicpf1}) have trivial intersection. Hence, we have
\begin{equation}\label{conicpf3}
(\Z/2\Z)^{\oplus N_{j}}\subseteq\Gamma^{2j/2j+1}(X)_{\tors}, \text{ where } N_{j}=\sum_{m=2j+1}^{n}{{n}\choose{k}}.
\end{equation}

Let $\beta_{i}=|\Gamma^{i/i+1}(X_{E})/\Im(\res^{i/i+1})|/|K^{i}(X_{E})/K^{i}(X)|$, where $K^{i}(X_{E})$ (resp. $K^{i}(X)$) is the codimension $i$ part of $K(X_{E})$ (resp. $K(X)$). Then, it follows from the base of $K(X)$ that $\beta_{1}, \beta_{2}\leq 1$. Since $c_{2}(2x_{i_{1}}x_{i_{2}})\cdots c_{2}(2x_{i_{2j-1}}x_{i_{2j}})=2^{j}y_{i_{1}}\cdots y_{i_{2j}}\in \Gamma^{2j}(X)$ and $2^{j+1}y_{i_{1}}\cdots y_{i_{2j+1}}\in \Gamma^{2j+1}(X)$, we obtain
\[\beta_{i}\leq 2^{[i+1/2]{{n}\choose{i}}}/2^{{{n}\choose{i}}}\]
for any $3\leq i\leq n$. Hence, by (\ref{alphalem}) we have
\begin{equation}\label{conicpf4}
|\!\oplus \Gamma^{i/i+1}(X)_{\tors}|\leq 2^{\sum_{i=3}^{n}([i+1/2]-1){{n}\choose{i}}}.
\end{equation}
Therefore, it follows from (\ref{conicpf3}) and (\ref{conicpf4}) that for any $j\geq 1$ we have $\Gamma^{2j/2j+1}(X)_{\tors}=(\Z/2\Z)^{N_{j}}$. As $N_{1}=N$, the result follows from (\ref{generic}).\end{proof}

\begin{corollary}\label{maximaltorsionconicdetermined}
For any $n\geq 2$, we have $\cM(\C_{n})=\cM(\PQ_{n})=2^{N}$.
\end{corollary}
\begin{proof}
Let $X=\prod_{i=1}^{n}\SB(Q_{i})\in \C_{n}$ such that $\ind(Q_{i})=2$ for all $i$. Given integers $1\leq m\leq n$ and $1\leq i_{1}<\cdots <i_{m}\leq n$, we define $d_{i_{1}\ldots i_{m}}=\max\{ \ind(Q_{i_{1}}^{\tens j_{1}}\tens \cdots \tens Q_{i_{m}}^{\tens j_{m}})\}$, where the maximum ranges over $0\leq j_{1},\ldots, j_{m}\leq 1$. If $m\geq 3$, then it follows from (\ref{gammatwo}) that $d_{i_{1}\ldots i_{m}}y_{i_{1}}\cdots y_{i_{m-1}}\in T^{2}(X)$. As $2y_{i_{m}}\in T^{1}(X)$, we have 
\begin{equation}\label{corollaryequforconics}
2d_{i_{1}\ldots i_{m}}y_{i_{1}}\cdots y_{i_{m}}\in T^{3}(X).
\end{equation}
Since the subgroup $T^{3}(X_{E})\cap K(X)$ is generated by $d_{i_{1}\ldots i_{m}}y_{i_{1}}\cdots y_{i_{m}}$ for all $3\leq m\leq n$, it follows from (\ref{tolchow}), (\ref{gammatwo}), and (\ref{corollaryequforconics}) that $|\CH^{2}(X)_{\tors}|\leq 2^{N}$. If $\ind(Q_{i})=1$ for some $i$, then by (\ref{reductionSB}) the computation of the upper bound for the torsion subgroup reduces to the case of product of less than $n$ conics. Hence, $\cM(\C_{n})\leq 2^{N}$, thus by Proposition \ref{lowerboundofconics} we obtain the result.\end{proof}

\subsection{Four Conics}

Let $Q_{i}$ be a quaternion algebra over a field $F$ for all $i\in I_{4}$. Consider the product $X$ of the corresponding conics $\SB(Q_{i})$. If $\ind(Q_{i})=1$ for some $i$, then by (\ref{reductionSB}) the problem to find torsion in $\CH^{2}(X)$ is reduced to the case of product of three conics (Proposition \ref{threeconics}). Hence, we may assume that $\ind(Q_{i})=2$ for all $i$. Let $d_{ij}=\ind(Q_{i}\tens Q_{j})$, $d_{i}=\ind(Q_{j}\tens Q_{k}\tens Q_{l})$, and $d=\ind(Q_{1}\tens Q_{2}\tens Q_{3}\tens Q_{4})$ for all distinct integers $i$, $j$, $k$, $l\in I_{4}$. For the same reason, we may assume that $d_{ij}$, $d_{i}$, $d\geq 2$. By (\ref{Quillenbasis}), we have the following basis of $K(X)$ of $16$ elements
\begin{equation}\label{basisfourconicsx}
\{1,\, 2x_{i}, d_{ij}x_{i}x_{j},\, d_{i}x_{j}x_{k}x_{l},\, dx_{1}x_{2}x_{3}x_{4}\},
\end{equation}
where $x_{i}$ is the pullback of the class of the tautological line bundle on the projective line for all $i\in I_{4}$.
By substitution $y_{i}=x_{i}-1$ for all $i\in I_{4}$, it follows from (\ref{basisfourconicsx}) that we have another basis $K(X)$
\begin{equation}\label{basisfourconicsy}
\{1,\, 2y_{i}, d_{ij}y_{i}y_{j},\, d_{i}(y_{j}y_{k}y_{l}+y_{j}y_{k}+y_{j}y_{l}+y_{k}y_{l}),\, d(y_{1}y_{2}y_{3}y_{4}+\sum y_{p}y_{q}y_{r}+\sum y_{p}y_{q})\},
\end{equation}
where the sums range over all $1\leq p<q<r\leq 4$.

We will frequently use the following lemma to find bounds for the torsion in $\CH^{2}(X)$.

\begin{lemma}\label{fourconicslem}
Let $i, j, k, l$ be distinct integers in $I_{4}$. Then, in codimension $2$, $3$, and $4$ respectively, we have
\smallskip

\noindent $(1)\, \Im(\res^{2/3})\ni
\begin{cases}
2\sum y_{p}y_{q} & \text{ if } d=2,\\
2(y_{i}y_{j}+y_{i}y_{k}+y_{j}y_{k}) & \text{ if } d_{l}=2,\\
\end{cases}
$

\medskip

\noindent $(2)\, \Im(\res^{3/4})\ni
\begin{cases}
4y_{i}y_{j}y_{k},\, 4y_{i}y_{j}y_{l} & \text{ if } d_{ij}=2,\\
4y_{i}y_{j}y_{k}, -4y_{i}y_{j}y_{k}+4\sum y_{p}y_{q}y_{r} & \text{ if } d_{l}=2,\\
4y_{i}y_{j}y_{k},\,4y_{i}y_{j}y_{l},\,4y_{i}y_{k}y_{l} & \text{ if } d_{ij}=d_{ik}=2.
\end{cases}$

\medskip

\noindent $(3)\, \Gamma^{4}(X)\ni
\begin{cases}
8y_{1}y_{2}y_{3}y_{4} & \text{ if } d=2,4 \text{ or } d_{ij}=2 \text{ or } d_{i}=2,\\
4y_{1}y_{2}y_{3}y_{4} & \text{ if } d_{ij}=d_{kl}=2,
\end{cases}$

\smallskip

\noindent where the sums range over all $1\leq p<q<r\leq 4$.
\end{lemma}
\begin{proof}
If $d=2$, then we have
\begin{equation}\label{fourconicsneweqd2}
c_{2}(2x_{1}x_{2}x_{3}x_{4})=2(y_{1}y_{2}y_{3}y_{4}+\sum y_{p}y_{q}y_{r}+\sum y_{p}y_{q})\in \Gamma^{2}(X),
\end{equation}
thus, we obtain $c_{1}(2x_{i})c_{1}(2x_{j})c_{2}(2x_{1}x_{2}x_{3}x_{4})=8y_{1}y_{2}y_{3}y_{4}\in \Gamma^{4}(X)$, which imply the results in Lemma \ref{fourconicslem} $(1)$, $(3)$, respectively. If $d_{l}=2$, then it follows from (\ref{gammatwo}) and (\ref{basisfourconicsy}) that
\begin{equation}\label{fourconicslemeq1}
2(y_{i}y_{j}y_{k}+y_{i}y_{j}+y_{i}y_{k}+y_{j}y_{k})\in \Gamma^{2}(X),
\end{equation}
which completes the proof of Lemma \ref{fourconicslem} $(1)$. Multiplying (\ref{fourconicslemeq1}) by $c_{1}(2x_{k})$ and $c_{1}(2x_{l})$, respectively, we obtain the results in Lemma \ref{fourconicslem} $(2)$. Multiplying $4y_{i}y_{j}y_{k}$ by $c_{1}(2x_{l})$, we have the result in Lemma \ref{fourconicslem} $(3)$.

If $d_{ij}=2$ for some $i$ and $j$, then it follows from (\ref{basisfourconicsy}) and (\ref{gammatwo}) that 
\begin{equation}\label{fourconicslemeq2}
2y_{i}y_{j}c_{1}(2x_{k}),\, 2y_{i}y_{j}c_{1}(2x_{l})\in \Gamma^{3}(X). 
\end{equation}
Multiplying the first element in (\ref{fourconicslemeq2}) by $c_{1}(2x_{l})$, we obtain the result in Lemma \ref{fourconicslem} $(3)$. If in addition $d_{kl}=2$, then, again by (\ref{gammatwo}) we obtain $2y_{i}y_{j}\cdot (2y_{k}y_{l})=4y_{1}y_{2}y_{3}y_{4}\in \Gamma^{4}(X)$. If $d=4$, then we have $c_{4}(4x_{1}x_{2}x_{3}x_{4})=24y_{1}y_{2}y_{3}y_{4}\in \Gamma^{4}(X)$. As $16y_{1}y_{2}y_{3}y_{4}\in \Gamma^{4}(X)$, we complete the proof of Lemma \ref{fourconicslem} $(3)$. The case $d_{ij}=d_{ik}=2$ in $(2)$ follows from the previous result for $d_{ij}=2$, which completes the proof of Lemma \ref{fourconicslem} $(2)$.\end{proof}

\begin{corollary}\label{fourconicslem2}
For all $1\leq p< q< r\leq 4$, $4y_{p}y_{q}y_{r}\in \Im(\res^{3/4})$ if one of the following conditions holds$:$ for some distinct $i, j, k, l\in I_{4}$
\begin{itemize}
\item $d=2$
\item $d_{ij}=d_{kl}=2$
\item $d_{ij}=d_{i}=2$
\item $d_{i}=d_{j}=d_{k}=2$
\item $d_{ij}=d_{ik}=d_{jk}=2$
\end{itemize}
\end{corollary}
\begin{proof}
Assume $d=2$. Then, by (\ref{fourconicsneweqd2}), we obtain
$$c_{1}(2x_{i})c_{2}(2x_{1}x_{2}x_{3}x_{4})=12y_{1}y_{2}y_{3}y_{4}+4y_{i}(y_{j}y_{k}+y_{j}y_{l}+y_{k}y_{l})\in \Gamma^{3}(X) \text { and }$$ 
\[c_{1}(2x_{1}x_{2}x_{3}x_{4})c_{2}(2x_{1}x_{2}x_{3}x_{4})=72y_{1}y_{2}y_{3}y_{4}+12\sum y_{p}y_{q}y_{r}\in \Gamma^{3}(X).\]
As $8\sum y_{p}y_{q}y_{r}\in \Gamma^{3}(X)$, the result follows. The remaining cases follow from Lemma \ref{fourconicslem} $(2)$.\end{proof}
\begin{remark}
If the number of $d_{ij}x_{i}x_{j}$'s in the basis (\ref{basisfourconicsx}) such that $d_{ij}=2$ is at least $4$, then by the second condition of Corollary \ref{fourconicslem2} we also have $4y_{p}y_{q}y_{r}\in \Im(\res^{3/4})$ for all $1\leq p< q< r\leq 4$.\end{remark}

We determine the codimension $2$ cycles of the product of $4$ conics as in Proposition \ref{threeconics}.
\begin{theorem}\label{fourconicsthm}
The torsion in Chow group of codimension $2$ of the product $X$ of $4$ conics is trivial if it satisfies one of the conditions $(\ref{fourconicsthmcase1})$, $(\ref{fourconicsthmcase2})$, $(\ref{fourconicsthmcase3})$, $(\ref{fourconicsthmcase4})$, $(\ref{fourconicsthmcase6})$, $(\ref{fourconicsthmcase8})$, $(\ref{fourconicsthmcase7})$,   $d=16$. Otherwise, the group $\CH^{2}(X)_{\tors}$ admits all elementary abelian group whose order is a divisor of $2^{5}$.
\end{theorem}
\begin{proof}
Let $Q_{i}$ be a quaternion division algebra over $F$ and $X$ the product of the corresponding conics $\SB(Q_{i})$ for $i\in I_{4}$. Set \[|K^{i}(X_{E})/K^{i}(X)|\cdot \beta_{i}=|\Gamma^{i/i+1}(X_{E})/\Im(\res^{i/i+1})|,\] where $E$ is a splitting field of $X$ and $K^{i}(X_{E})$ (resp. $K^{i}(X)$) is the codimension $i$ part of $K(X_{E})$ (resp. $K(X)$). We find upper bounds of $\beta_{i}$ (equivalently, $|\Gamma^{i/i+1}(X_{E})/\Im(\res^{i/i+1})|$) using case by case analysis. Note that we have $\beta_{1}\leq 1$.

For convenience of the reader we set up some notations. For $n=2, 4, 8$, we set
\[ D_{n}=|\{1\leq i\leq 4 \,|\, d_{i}=n\}|\quad\text{ and }\quad D=|\{ij\in \{12,\, 13,\, 14,\, 23,\, 24,\, 34\}\,|\, d_{ij}=2\}|. \]
We begin with some observations on $d$: if $d=2^{4}$, then we have $d_{i}=d/2$ for all $i$, which implies that $d_{ij}=d/4$ for all $1\leq i\neq j\leq 4$. Therefore, we obtain $\beta_{i}\leq 1$ for all $i$, thus, $|\oplus\Gamma^{i/i+1}(X)_{\tors}|\leq 1$, i.e., the group $\CH^{2}(X)_{\tors}$ is trivial. Hence, we may assume that $d=2, 4, 8$. Note also that we have $d=2$ or $4$ $\text{ if } D_{2}\geq 1, \text{ and } D_{8}=0 \text{ if } d=2$, thus we only consider the cases where $d=4, 8$ (resp. $d=2, 4$) in the first $3$ cases (resp. the last $4$ cases) of the following. We divide the proof into  $8$ cases according to $D_{n}$.


\framebox{{\it Case}: $D_{8}\geq 3$.} In this case, we have $d_{ij}=4$ for all $i\neq j$ (i.e., $D=0$) and $0\leq D_{2}\leq 1$. It follows from Lemma \ref{fourconicslem} $(1)$ that $4^{6}\cdot \beta_{2}\leq 4^{5}\cdot 2^{2-D_{2}}$. More precisely, if $d_{i}=2$, we first replace one of $y_{j}y_{k}, y_{j}y_{l}, y_{k}y_{l}$ in codimension $2$ part of basis of $K(X_{E})$ by $y_{j}y_{k}+y_{j}y_{l}+y_{k}y_{l}$ and then apply Lemma \ref{fourconicslem} $(1)$. Unless otherwise specified, for the rest of the proof we apply the same trick when we use Lemma \ref{fourconicslem} and Corollary \ref{fourconicslem2} without mentioning this. By Lemma \ref{fourconicslem} $(2)$, we obtain
\[\beta_{3}\leq
(8^{2}\cdot 4^{2})/8^{3}\cdot 2\,\, \text{ (resp. }\leq 2^{D_{4}}) \text{ if } D_{2}=1 \text{ (resp. otherwise)}.\]
It follows by Lemma \ref{fourconicslem} $(3)$ that $\beta_{4}\leq 2$. Hence, by (\ref{alphalem}), (\ref{secondstatementone}), and (\ref{secondstatementtwo})  the order of the group $\oplus\Gamma^{i/i+1}(X)_{\tors}$ is nontrivial except the case where
\begin{equation}\label{fourconicsthmcase1}
D_{2}=1, D_{8}=3, d=4.
\end{equation}

In the remaining cases, if $D_{8}=4$, $d=8$ (resp. $d=4$), then $\Gamma^{3}(X)$ is generated by $\{16y_{1}y_{2}y_{3}y_{4}, 8y_{i}y_{j}y_{k}\}$ (resp. $\{8y_{1}y_{2}y_{3}y_{4}, 8y_{i}y_{j}y_{k}\}$) for all distinct $i, j, k\in I_{4}$ over $K(X)$, thus by the divisibility of elements in $\Gamma^{3}(X)$ the class of $8y_{1}y_{2}y_{3}y_{4}\,  (\text{resp. } 4y_{1}y_{2}y_{3}y_{4})\in \Gamma^{2}(X)\backslash \Gamma^{3}(X)$ gives a torsion element of $\Gamma^{2/3}(X)$ of order $2$. As $\beta_{1}\beta_{2}\beta_{3}\beta_{4}\leq 2$, we have 
\begin{equation}\label{secondstatementone}
\Gamma^{2/3}(X)_{\tors}=\Z/2\Z.
\end{equation}
Similarly, if $D_{8}=3$, $d_{i}=4$, and $d=4$ (resp. $d=8$) for some $i$, then we obtain 
\begin{equation}\label{secondstatementtwo}
\Gamma^{2/3}(X)_{\tors}=(\Z/2\Z)^{\oplus 2}\,\,\,\, \text{(resp. }\Gamma^{2/3}(X)_{\tors}=\Gamma^{3/4}(X)_{\tors}=\Z/2\Z)
\end{equation}
generated by the classes of $4y_{j}y_{k}y_{l}$ and $4(y_{i}y_{j}y_{k}+y_{i}y_{j}y_{l}+y_{i}y_{k}y_{l}+y_{1}y_{2}y_{3}y_{4})$ (resp. $8y_{1}y_{2}y_{3}y_{4}$).

\framebox{{\it Case}: $D_{8}=2$.} A simple calculation of index implies that $0\leq D\leq 1$. Hence, it follows from Lemma \ref{fourconicslem} $(1)$ that
\begin{equation}\label{fourconicsthmcase2beta2}
(4^{6-D}\cdot 2^{D}) \beta_{2}\leq 4^{6-D-D_{2}}\cdot 2^{D+D_{2}}, \text{ i.e., }\,\beta_{2}\leq 1/2^{D_{2}}. 
\end{equation}

Assume $D=0$. If $D_{2}=0$, then $\beta_{3}\leq 8^{4}/(8^{2}\cdot 4^{2})$. Otherwise, by Lemma \ref{fourconicslem} $(2)$ we obtain
\[(8^{2}\cdot 4^{2-D_{2}}\cdot 2^{D_{2}}) \beta_{3}\leq 8^{3-D_{2}}\cdot 4^{1+D_{2}} \quad \text{ for }D_{2}=1,2.\]

Now we assume $D=1$. If $D_{2}=0$, then by Lemma \ref{fourconicslem} $(2)$,  $(8^{2}\cdot 4^{2})\beta_{3}\leq 8^{2}\cdot 4^{2}$. Otherwise, by Corollary \ref{fourconicslem2} and Lemma \ref{fourconicslem} $(2)$ respectively, for $D_{2}=1,2$ one has
\[(8^{2}\cdot 4^{2-D_{2}}\cdot 2^{D_{2}}) \beta_{3}\leq
\begin{cases}
4^{4} & \text{ if } d_{i}=d_{ij}=2 \text{ for some } 1\leq i\neq j\leq 4,\\
4^{3}\cdot 8 & \text{ otherwise.} 
\end{cases}\]

It follows from Lemma \ref{fourconicslem} $(3)$ that \begin{equation}\label{fourconicsthmcase2beta4}\beta_{4}\leq 1\, \text{ (resp.}\leq 2) \text{ if } D\geq1, D_{2}=0, d=8 \text{ (resp. otherwise)}.\end{equation}
Therefore, it follows by the same argument as in the previous case that $|\!\oplus\Gamma^{i/i+1}(X)_{\tors}|$ is trivial in the following cases:
\begin{equation}\label{fourconicsthmcase2}
D_{2}=2; \text{ or } D=D_{2}=1; \text{ or } D=1, D_{2}=0, d=8.
\end{equation}

\framebox{{\it Case}: $D_{8}=1$.} By this assumption, we have $0\leq D\leq 3$. Note also that if $d_{ij}=d_{kl}=2$ or $d_{ij}=d_{ik}=d_{jk}=2$ for some distinct $i, j, k, l\in I_{4}$, then $D_{8}=0$. It follows by the same argument as in the previous case that
we have the same upper bounds in (\ref{fourconicsthmcase2beta2}) and (\ref{fourconicsthmcase2beta4}) for $\beta_{2}$ and $\beta_{4}$, respectively. 

Assume $D_{2}=3$. Then, by Corollary \ref{fourconicslem2} one has $\beta_{3}\leq 4^{4}/(8\cdot 2^{3})=4$. From now on we only consider the remaining cases: $0\leq D_{2}\leq 2$. If $d_{i}=d_{ij}=2$ (thus, $D, D_{2}\geq 1$) for some $i, j$, then it follows from Corollary \ref{fourconicslem2} that $$(8\cdot 4^{3-D_{2}}\cdot 2^{D_{2}}) \beta_{3}\leq 4^{4}.$$
Otherwise, by Lemma \ref{fourconicslem} $(2)$ we have
\[(8\cdot 4^{3-D_{2}}\cdot 2^{D_{2}}) \beta_{3}\leq
\begin{cases}
8^{4} & \text{ if } D_{2}=D=0,\\
8^{2}\cdot 4^{2} & \text{ if } D_{2}+D=1,\\
8\cdot 4^{3} & \text{ if } D_{2}+D=2, 3.
\end{cases}\]

By the same argument, we conclude that the order of the torsion $|\!\oplus\Gamma^{i/i+1}(X)_{\tors}|$ is trivial in the following cases: for some distinct $i, j\in I_{4}$
\begin{equation}\label{fourconicsthmcase3}
D_{2}=0, D\geq 2, d=8 ; \text{ or } 1\leq D_{2}\leq 2, D=1, d_{i}=d_{ij}=2; \text{ or } 1\leq D_{2}\leq 2, D\geq 2; \text{ or }  D_{2}=3.\end{equation}

If $D=6$, $d=8$, and $d_{i}=8$, then by the same argument as above we obtain 
\begin{equation}\label{secondstatementthree}
\Gamma^{2/3}(X)_{\tors}=(\Z/2\Z)^{\oplus 3}\text{ and }\Gamma^{3/4}(X)_{\tors}=\Z/2\Z
\end{equation}
 generated by $4y_{i}y_{k}y_{l}$, $4y_{i}y_{j}y_{l}$, $4y_{i}y_{j}y_{k}$, and $8y_{1}y_{2}y_{3}y_{4}$, respectively.


\framebox{{\it Case}: $D_{4}=4$.} By Lemma \ref{fourconicslem} $(1)$, it immediately follows that
\[\beta_{2}\leq 1/2\,\,\, (\text{resp.}\leq 1 ) \text{ if } d=2, D\neq 6 \,\,\,(\text{resp. otherwise}).\]

Similarly, it follows from Corollary \ref{fourconicslem2} and Lemma \ref{fourconicslem} $(2)$ that
\[4^{4}\cdot \beta_{3}\leq
\begin{cases}
4^{4}& \text{ otherwise,} \\
8^{2}\cdot 4^{2} \,\,(\text{resp. }8^{4})& \text{ if } d\neq 2, D=1 \,\,(\text{resp. }D=0),\\
8\cdot 4^{3}& \text{ if } d\neq 2, d_{ij}=d_{ik}=2, d_{jk}=d_{jl}=d_{kl}=4 
\end{cases}
\]
for some $i, j, k, l$. By Lemma \ref{fourconicslem} $(3)$ and (\ref{basisfourconicsy}), we obtain
\[\beta_{4}\leq
\begin{cases}
1 & \text{ if }d=8, D\geq 1; \text{ or } d=4, d_{ij}=d_{kl}=2 \text{ for some }  i, j, k, l\in I_{4}, \\
2 & \text{ otherwise,}\\
2^{2}& \text{ if }d=2, d_{ij}+d_{kl}\neq 4 \text{ for all }  i, j, k, l.
\end{cases}
\]
Applying the same argument, we see that the order of $\oplus\Gamma^{i/i+1}(X)_{\tors}$ is trivial in the following cases: for some distinct $i, j, k, l\in I_{4}$
\begin{equation}\label{fourconicsthmcase4}
d\neq 2, d_{ij}=d_{kl}=2; \text{ or } d=4, d_{ij}=d_{ik}=d_{jk}=2; \text{ or } d=2, D\neq 6, d_{ij}=d_{kl}=2.
\end{equation}

If $D=0$ and $d=4$, then by (\ref{basisfourconicsy}) and (\ref{gammatwo}) one has $d_{i}y_{j}y_{k}y_{l}$, $4y_{1}y_{2}y_{3}y_{4}\in\Gamma^{2}(X)\backslash \Gamma^{3}(X)$ for all $i, j, k, l$ since any element of $\Gamma^{3}(X)$ is divisible by $2^{3}$. Therefore, by (\ref{basisfourconicsy}) and Lemma \ref{fourconicslem} $(3)$ the classes of $d_{i}y_{j}y_{k}y_{l}$ and $4y_{1}y_{2}y_{3}y_{4}$ give torsion elements of $\Gamma^{2/3}(X)$ of order $2$. Moreover, the subgroups generated by these classes have trivial intersection by the divisibility of an element of $\Gamma^{3}(X)$. Hence,
it follows from Corollary \ref{maximaltorsionconicdetermined} that
\begin{equation}\label{fourconicsthmcase4prime}
\Gamma^{2/3}(X)_{\tors}=(\Z/2\Z)^{\oplus 5}, \text{ thus } \CH^{2}(\bar{X})_{\tors}=(\Z/2\Z)^{\oplus 5},
\end{equation} 
where $\bar{X}$ is the corresponding generic variety. If $D=0$ and $d=8$, then by the same argument as above the classes of $d_{i}y_{j}y_{k}y_{l}$ generate different subgroups of $\Gamma^{2/3}(X)_{\tors}$ of order $2$ and the class of $8y_{1}y_{2}y_{3}y_{4}=2y_{1}(4y_{2}y_{3}y_{4})$ generates a subgroup of $\Gamma^{3/4}(X)_{\tors}$ of order $2$. Therefore, it follows Corollary \ref{maximaltorsionconicdetermined} that
\begin{equation}\label{fourconicsthmcase4prime2}
\Gamma^{2/3}(X)_{\tors}=(\Z/2\Z)^{\oplus 4} \text{ and } \Gamma^{3/4}(X)_{\tors}=\Z/2\Z.
\end{equation}


\framebox{{\it Case}: $D_{2}=4$.} In codimension $2$, if $D\geq 4$, then by the assumption $D_{2}=4$ and Lemma \ref{fourconicslem} $(1)$ we have $2y_{i}y_{j}\in \Im(\res^{2/3})$ for all $i, j$, thus $(4^{6-D}\cdot 2^{D})\cdot \beta_{2}\leq 2^{6}$, i.e., $\beta_{2}\leq 1/2^{6-D}$. Let $z_{i}=y_{j}y_{k}+y_{j}y_{l}+y_{k}y_{l}$ for all $i$. If $D=0$ and $d=2$ (resp. $d=4$), then we replace the codimension $2$ part of the basis of $K(X_{E})$ by $$\{z_{i}, z_{j}, z_{k}, \sum_{1\leq p<q\leq 4} y_{p}y_{q}, y_{j}y_{k}, y_{j}y_{l}\} \text{ (resp. } \{z_{i}, z_{j}, z_{k}, y_{j}y_{k}, y_{j}y_{l}, y_{i}y_{l}\})$$ for any $i, j, k, l$. Hence, by Lemma \ref{fourconicslem} $(1)$ we obtain $\beta_{2}\leq d/2^{5}$. By the same argument, we easily see that the same inequality holds for $D=1, 2$. For the remaining case $D=3$, by the same trick and Lemma \ref{fourconicslem} $(1)$ we obtain 
\[2^{9}\cdot \beta_{2}\leq  4\cdot 2^{5}\text{ (resp. }\leq 2^{6})\,\,\, \text{ if } d=4, d_{ij}=d_{ik}=d_{jk}=2 \text{ (resp. otherwise})\] for some $i, j, k$. In codimension $3$, as $D_{2}=4$, it follows from Corollary \ref{fourconicslem2}  that $2^{4}\cdot \beta_{3}\leq 4^{4}$.

In codimension $4$, by Lemma \ref{fourconicslem} $(3)$, we have 
\begin{equation}\label{case5beta_4}\beta_{4}\leq 4/d\, \text{ (resp. }\leq 8/d)\,\,\, \text{ if } d_{ij}=d_{kl}=2 \text{ for some }i, j, k, l \text{ (resp. otherwise)}.\end{equation}
Observe that the same upper bound (\ref{case5beta_4}) of $\beta_{4}$ also works for the remaining three cases: $D_{2}=3, D_{4}=1;\, D_{2}=2, D_{4}=2;\, D_{2}=1, D_{4}=3$.

Applying the same argument together with the upper bounds $\beta_{i}$, we have $\Gamma^{2/3}(X)_{\tors}\neq 0$ in any case. In particular, if $D=6$ and $d=2$, then by Proposition \ref{lowerboundofconics} we have
\begin{equation}\label{fourconicsthmcase5prime}
\Gamma^{2/3}(X)_{\tors}=(\Z/2\Z)^{\oplus 5}, \text{ thus } \CH^{2}(\bar{X})_{\tors}=(\Z/2\Z)^{\oplus 5}.
\end{equation}


\framebox{{\it Case}: $D_{2}=3$, $D_{4}=1$.} Observe that in the previous case we only use at most three $z_{i}$'s for the new basis in codimension $2$. Thus, one easily see that $D_{2}=3$ does not affect on the upper bounds of $\beta_{2}$ in the previous case, i.e., we have the same upper bound for $\beta_{2}$ as in the previous case. In codimension $3$, as $D_{2}=3$ and $D_{4}=1$, it follows from Corollary \ref{fourconicslem2}  that $(2^{3}\cdot 4) \beta_{3}\leq 4^{4}$. 

Using (\ref{case5beta_4}) and the same argument as in the previous case, we see that $|\!\oplus\Gamma^{i/i+1}(X)_{\tors}|$ is trivial in the following cases: for some distinct $i, j, k, l\in I_{4}$
\begin{equation}\label{fourconicsthmcase6}
D\leq 3,\, d_{ij}=d_{kl}=2.
\end{equation}



\framebox{{\it Case}: $D_{2}=1$, $D_{4}=3$.} Let $d_{l}=2$ and $d_{i}=d_{j}=d_{k}=4$. We first find upper bounds for $\beta_{2}$ based on $D$. Obviously, we have $\beta_{2}\leq 1$ if $D=6$. If $D=5$, then by Lemma \ref{fourconicslem} $(1)$
\[(4\cdot 2^{5})\beta_{2}\leq 2^{6}\text{ (resp. }\leq 4\cdot 2^{5})\,\,\, \text{ if }d=2 \text{ or } d_{ij}+d_{ik}+d_{jk}=8 \text{ (resp. otherwise). }\] 
If $D=0,1,2$, then according to Lemma \ref{fourconicslem} $(1)$ the upper bound depends on $d$ as follows (we always use $y_{ij}+y_{ik}+y_{jk}$ as a part of basis of $K(X_{E})$. When $d=2$, we also use $\sum_{1\leq p<q\leq 4}y_{p}y_{q}$ as a part of basis of $K(X_{E})$): 
\[(4^{6-D}\cdot 2^{D})\beta_{2}\leq 2^{D}\cdot 2\cdot 4^{5-D} \text{ (resp. }\leq 2^{D}\cdot 2\cdot (2\cdot 4^{4-D}))\,\,\, \text{ if }d=2 \text{ (resp. otherwise). }\] 
For the remaining case $D=3, 4$, by the same argument we obtain
\begin{equation*}
	\beta_{2}\leq
	\begin{cases}
		1/2^{2} & \text{ if } d_{ij}=4, d_{ik}=d_{jk}=d=2; \text{ or } d_{ij}=d_{ik}=4, d_{jk}=d=2, D=3,\\
		1/2 & \text{ otherwise,}\\
		1 & \text{ if } d_{ij}=d_{ik}=d_{jk}=2, d=4.
	\end{cases}
\end{equation*}

In codimension $3$, it follows from Lemma \ref{fourconicslem} $(2)$ that
\begin{equation*}
(4^{3}\cdot 2) \beta_{3}\leq 8^{2}\cdot 4^{2}\text{ (resp. }\leq 8\cdot 4^{3})\,\,\, \text{ if } D=0, d=4 \text{ (resp. }D=1, d_{ij}=2, d=4).
\end{equation*}
Otherwise, by Corollary \ref{fourconicslem2}  we obtain $(4^{3}\cdot 2)\beta_{3}\leq 4^{4}$.


Combining all bounds of $\beta_{i}$ together with (\ref{case5beta_4}), we see that $|\!\oplus\Gamma^{i/i+1}(X)_{\tors}|$ is trivial in the following cases: for some distinct $i, j, k, l\in I_{4}$
\begin{align}\label{fourconicsthmcase8}
&D=5, d_{l}=2, d_{ij}+d_{ik}+d_{jk}\neq 6, d=4; \text{ or } D=4, d_{l}=2, d_{ij}=4, d_{ik}=2, d_{jk}=2; \\
& \text{or }D=4, d_{l}=d_{jk}=2, d_{ij}=d_{ik}=d=4; \text{ or }D\leq 3, d_{ij}=d_{kl}=2 .\nonumber
\end{align}


\framebox{{\it Case}: $D_{2}=2$, $D_{4}=2$.} Let $d_{k}=d_{l}=2$ and $d_{i}=d_{j}=4$. The upper bounds for $\beta_{2}$ is slightly different from the previous bounds: the only difference comes from possible use of $y_{ij}+y_{il}+y_{jl}$ as a part of basis of $K(X_{E})$. If $D=6$, then $\beta_{2}\leq 1$. If $D=5$, then it follows by Lemma \ref{fourconicslem} $(1)$ that 
\[(4\cdot 2^5) \beta_{2}\leq 4\cdot 2^{5}\text{ (resp. }\leq 2^{6})\,\,\, \text{ if }d=d_{kl}=4\text{ (resp. otherwise). }\] 
For the remaining case $D=3, 4$, we always use $y_{ij}+y_{ik}+y_{jk}$ and $y_{ij}+y_{il}+y_{jl}$ as a part of basis of $K(X_{E})$ (If in addition $d=2$, we also use $\sum_{1\leq p<q\leq 4}y_{p}y_{q}$ as a part of the basis as before). When $D=4$, by Lemma \ref{fourconicslem} $(2)$ we have
\[\beta_{2}\leq 1/2^{2}\,\,\, \text{ if }d=2; \text{ or }d_{ij}=4, d_{kl}=2; \text{ or } d_{ij}=d_{kl}=2, d_{ik}+d_{jk}=6.\] 
Otherwise, $\beta_{2}\leq 1/2$. Similarly, if $D=3$, then again by Lemma \ref{fourconicslem} $(2)$
\begin{equation*}
	\beta_{2}\leq
	\begin{cases}
		1/2^{3}& \text{ if } d=2, d_{ij}=d_{kl}=4; \text{ or } d=2, d_{ij}+d_{kl}=d_{ik}+d_{jk}=6,\\
		1/2^{2} & \text{ otherwise,}\\
		1/2 & \text{ if } d=d_{kl}=4, d_{ij}=2, d_{ik}+d_{jk}\neq 6.
	\end{cases}
\end{equation*}

In codimension $3$, It follows from Lemma \ref{fourconicslem} $(2)$ and Corollary \ref{fourconicslem2} that 
\begin{equation*}(4^2\cdot 2^2) \beta_{3}\leq 8\cdot 4^{3}\, \text{ (resp. }\leq 4^{4}) \text{ if } D=0, d=4; \text{ or } D=1, d_{ij}=2, d=4 \text{ (resp. otherwise)}.\end{equation*}

By (\ref{case5beta_4}) and the same argument as in the previous case, the torsion $\!\oplus\Gamma^{i/i+1}(X)_{\tors}$ is trivial in the following cases: for some distinct $i, j, k, l\in I_{4}$
\begin{align}\label{fourconicsthmcase7}
	&D=2, d_{ij}=d_{kl}=2, d=4; \text{ or } D=d=4, d_{k}+d_{l}=d_{ij}+d_{kl}=4, d_{ik}+d_{jk}=6; \text{ or}\\\nonumber
	& D=3, d_{k}+d_{l}=4, d_{ij}+d_{kl}=8; \text{ or }D=3, d_{k}+d_{l}=d_{ij}+d_{kl}=d=4; \text{ or } \\\nonumber
	& D=3, d_{k}+d_{l}=4, d_{ij}+d_{kl}=6, d_{ik}+d_{jl}\neq 6.\nonumber
\end{align}

Finally, the second statement of the theorem comes from (\ref{secondstatementone}), (\ref{secondstatementtwo}), (\ref{secondstatementthree}), (\ref{fourconicsthmcase4prime}), and (\ref{fourconicsthmcase4prime2}).\end{proof}

\begin{remark}\label{fourconicsremarkuse}
Indeed, as we have shown in the proof of the first case, one can show that the upper bounds $\beta_{1}\beta_{2}\beta_{3}\beta_{4}$ for each case of the proof of Theorem \ref{fourconicsthm} are sharp.
\end{remark}

\subsection{Galois cohomology and torsion groups}\label{applicationgalcoh}

The torsion subgroup of the Chow group of codimension $2$ cycles can be used to measure how far is a relative Galois cohomology group from being a decomposable subgroup (generated by the class of $A_{i}$ below) \cite[Theorem 4.1]{Pey}. Namely, for an $F$-variety $X=\prod_{i} \SB(A_{i})\in \C_{n}$ or $\SBS_{n}$ we have
\begin{equation}\label{applicationchowtwo}
\CH^{2}(X)_{\tors}\simeq H^{3}(F(X)/F, \Q/\Z(2))/\oplus_{i} H^{1}(F, \Q/\Z(1))\cup [A_{i}],
\end{equation}
where $H^{3}(F(X)/F, \Q/\Z(2))$ denotes the kernel of $H^{3}(F, \Q/\Z(2))\to H^{3}(F(X), \Q/\Z)$ of Galois cohomology groups with coefficient in $\Q/\Z(2)$ and $[A_{i}]$ denotes the class in the Brauer group $\Br(F)=H^{2}(F, \Q/\Z(1))$. Therefore, our main results (Corollary \ref{maximaltorsionconicdetermined}, Proposition \ref{lowerSB}, and Theorem \ref{threeproductSBThm}) tell us how large indecomposable subgroups we can have.

 Moreover, by \cite[Remark 4.1]{Pey} there is a canonical injection from a Galois cohomology group with the finite coefficient $\gmu_{n}^{\tens 2}$ to the torsion subgroup:
\begin{equation}\label{applicationapp}
H^{3}(F(X)/F, \gmu_{n}^{\tens 2})/\oplus_{i} H^{1}(F, \gmu_{n})\cup [A_{i}]\hookrightarrow \CH^{2}(X)_{\tors}.
\end{equation}
Therefore, if the torsion subgroup $\CH^{2}(X)_{\tors}$ is trivial, then one can write the relative Galois cohomology group in terms of decomposable subgroups with the finite coefficient $\gmu_{n}$. For instance, if $X\in C_{4}$ satisfying one of the conditions $(\ref{fourconicsthmcase1})$, $(\ref{fourconicsthmcase2})$, $(\ref{fourconicsthmcase3})$, $(\ref{fourconicsthmcase4})$, $(\ref{fourconicsthmcase6})$, $(\ref{fourconicsthmcase8})$, $(\ref{fourconicsthmcase7})$, $d=16$ in Theorem \ref{fourconicsthm}, then we obtain $|\CH^{2}(X)_{\tors}|=1$, thus by (\ref{applicationapp}) we have
\begin{equation*}
H^{3}(F(X)/F, \Z/2\Z)=\oplus_{i=1}^{4} H^{1}(F, \Z/2\Z)\cup [Q_{i}].
\end{equation*}

Finally, we mention another application of the results (Proposition \ref{threeconics} and Theorem \ref{fourconicsthm}) as follows. Consider a split reductive group $G$ over $F$ and its semisimple derived subgroup $G'$. Let $Y$ be a generic homogeneous variety $\bar{U}/B$ over $F(U/G')$, where $U$ is an open subset of a generically free representation of $G'$, $B$ is a Borel subgroup of $G'$, and $\bar{U}$ is the generic fiber of $U$ over $U/G'$. Then, the torsion part of $\CH^{2}(Y)$ determines the group of semi-decomposable invariants $\operatorname{Sdec}(G')$ and $\operatorname{Sdec}(G)$ of $G'$ and $G$ with coefficients in $\Q/\Z(2)$ \cite[Theorem]{MNZ}, \cite[Proposition 2.1]{Baek}:
\begin{equation*}
\CH^{2}(Y)_{\tors}\simeq \Inv^{3}(G')/\operatorname{Sdec}(G'), \,\, \operatorname{Ker}(\Inv^{3}(G)_{\operatorname{ind}}\to \CH^{2}(Y)_{\tors})\simeq \operatorname{Sdec}(G)/\operatorname{Dec}(G),
\end{equation*}
where $\Inv^{3}(G')$ is the group of degree $3$ normalized invariants of $G'$, $\Inv^{3}(G)_{\operatorname{ind}}$ is the group of indecomposable invariants of $G$, and $\operatorname{Dec}(G)$ is the subgroup of decomposable invariants of $G$.

In particular, when $G'$ is the quotient of $(\gSL_{2})^{n}$, $n\geq 4$, by its maximal central subgroup, the computation of $\CH^{2}(Y)_{\tors}$ is reduced to computing $\CH^{2}(\bar{X})_{\tors}$ and $\CH^{2}(\bar{X'})_{\tors}$, where $X=\prod_{i=1}^{3}\SB(Q_{i})$, $X'=\prod_{i=1}^{4}\SB(Q'_{i})$ associated to division quaternion algebras $Q_{i}, Q'_{i}$ such that $\ind(\prod_{i=1}^{4}Q'_{i})=\ind(\prod_{i=1}^{3}Q_{i})=2$, $\ind(Q'_{i}\tens Q'_{j}\tens Q'_{k})=\ind(Q'_{i}\tens Q'_{j})=\ind(Q_{p}\tens Q_{q})=4$ for all distinct $i, j, k\in I_{4}$, $p, q\in I_{3}$ \cite[Lemma 5.1, Proposition 5.2]{Baek}. These results are particular cases of Proposition \ref{threeconics} and Theorem \ref{fourconicsthm}, respectively and serve as the main ingredient in the proof of \cite[Theorem]{Baek}, which completely determines degree $3$ invariants of $G'$ and the corresponding reductive group $G$.

\section{Product of Severi-Brauer surfaces}\label{section5}

In this section, we find a general lower bound of the torsion in Chow group of codimension $2$ of the product of $n$ Severi-Brauer surfaces in Proposition \ref{lowerSB} and prove that the lower bound is sharp for $n=3$ in Theorem \ref{threeproductSBThm}. To prove Proposition \ref{lowerSB} we need the following:
\begin{lemma}\label{KeySB}
Let $p$ be an odd prime and let $n\geq 2$ and $m_{i}$ be integers such that $1\leq m_{i}\leq p-1$ for all $1\leq i\leq n$. Let $\Phi$ be the polynomial $(\prod_{i=1}^{n}(s_{i}+1)^{m_{i}}-1)^{p}$ in $\Z[s_{1}, \ldots, s_{n}]$. Then, the alternating sum $\sum_{j_{1}=\cdots=j_{n}=1}^{p-1}(-1)^{j_{1}+\cdots +j_{n}}C_{j_{1}\cdots j_{n}}$ in the quotient $\Z[s_{1},\ldots, s_{n}]/(s_{1}^{p},\ldots, s_{n}^{p})$ is divisible by $p^{2}$, where $C_{j_{1}\cdots j_{n}}$ is the coefficient of the monomial $s_{1}^{j_{1}}\cdots s_{n}^{j_{n}}$ in $\Phi$.
\end{lemma}
\begin{proof}
Let $t_{n}=(s_{n}+1)^{m_{n}}$. For any $1\leq j_{1}, \ldots, j_{n-1}\leq p-1$, we write $\sum_{k=1}^{p} d_{j_{1}\cdots j_{n-1}k}\cdot t_{n}^{k}$ for the coefficient of $s_{1}^{j_{1}}\cdots s_{n-1}^{j_{n-1}}$ in $\Phi$. Let $e_{j_{1}\cdots j_{n-1}}$ be the coefficient of $s_{1}^{j_{1}}\cdots s_{n-1}^{j_{n-1}}$ in $\Psi_{n}:=(s_{1}+1)^{pm_{1}}\cdots (s_{n-1}+1)^{pm_{n-1}}$. By expanding each factor $(s_{i}+1)^{pm_{i}}=((s_{i}+1)^{p})^{m_{i}}$ of $\Psi_{n}$, we see that
\begin{equation}\label{ejdivisible}
e_{j_{1}\cdots j_{n-1}}=\big({{p}\choose{j_{1}}}m_{1}+a_{1}\big)\cdots \big({{p}\choose{j_{n-1}}}m_{n-1}+a_{n-1}\big)
\end{equation}
for some integers $a_{1}, \ldots, a_{n-1}$, such that $p^{2}\mid a_{i}$ for all $1\leq i\leq n-1$. 

We prove by induction on $n$. Assume $n=2$. First, observe that $C_{j_{1}0}=0$ in the quotient $\Z[s_{1}, s_{2}]/(s_{1}^{p}, s_{2}^{p})$. Hence, we have
\[\sum_{j_{2}=1}^{p-1}(-1)^{j_{2}}C_{j_{1}j_{2}}\!=\!\!\sum_{j_{2}=0}^{p-1}(-1)^{j_{2}}C_{j_{1}j_{2}}\!=\!\!\sum_{k=1}^{p}d_{j_{1}k}\!\sum_{i=0}^{p-1}(-1)^{i}{{m_{2}k}\choose{i}}\!=\!\!\sum_{k=1}^{p}d_{j_{1}k}{{m_{2}k-1}\choose{p-1}},\]
which implies that
\begin{equation}\label{inductwo}
\sum_{j_{1}=j_{2}=1}^{p-1}(-1)^{j_{1}+j_{2}}C_{j_{1}j_{2}}=\sum_{j_{1}=1}^{p-1}\sum_{k=1}^{p}d_{j_{1}k}{{m_{2}k-1}\choose{p-1}}(-1)^{j_{1}}.
\end{equation}
For each $1\leq k\leq p-1$, we see that
\begin{equation}\label{indtworesult1}
p\mid {{m_{2}k-1}\choose{p-1}}.
\end{equation}
Moreover, as $\Phi=\sum_{i=0}^{p}{{p}\choose{i}}(s_{1}+1)^{m_{1}i}t_{2}^{i}(-1)^{p-i}$, we obtain
\begin{equation}\label{indtworesult2}
p\mid d_{j_{1}k}
\end{equation}
for any $1\leq k\leq p$. From (\ref{indtworesult1}) and (\ref{indtworesult2}), it suffices to show that 
\begin{equation}\label{divisibilityofntwo}
p^{2}\mid \sum_{j_{1}=1}^{p-1} d_{j_{1}p}{{m_{2}p-1}\choose{p-1}}(-1)^{j_{1}}.
\end{equation}
Since $d_{j_{1}p}=e_{j_{1}}$ and $\sum_{j_{1}=1}^{p-1}{{p}\choose{j_{1}}}(-1)^{j_{1}}=0$, the divisibility in (\ref{divisibilityofntwo}) follows from (\ref{ejdivisible}).

Now we assume that the result holds for $n-1$. By the induction hypothesis, it is enough to show that $\sum_{j_{1}=\cdots=j_{n-1}=1, j_{n}=0}^{p-1}(-1)^{j_{1}+\cdots+j_{n}}C_{j_{1}\cdots j_{n}}$ is divisible by $p^{2}$. Applying the same argument as in the case $n=2$, we have
\begin{equation}\label{inducn}
\sum_{j_{1}=\cdots=j_{n-1}=1, j_{n}=0}^{p-1}\!\!\!\!\!\!\!\!(-1)^{j_{1}+\cdots+j_{n}}C_{j_{1}\cdots j_{n}}=\!\!\!\!\!\!\!\sum_{j_{1}=\cdots=j_{n-1}=1}^{p-1}\sum_{k=1}^{p}d_{j_{1}\cdots j_{n-1}k}{{m_{n}k-1}\choose{p-1}}\!(-1)^{j_{1}+\cdots+j_{n-1}}.
\end{equation}
As before, we know $p\mid {{m_{n}k-1}\choose{p-1}}$ for each $1\leq k\leq p-1$. By the same argument with $t_{n}$, we also know $p\mid d_{j_{1}\cdots j_{n-1}k}$ for each $1\leq k\leq p$. For $k=p$, we have $p^{2}\mid d_{j_{1}\cdots j_{n-1}p}$ by (\ref{ejdivisible}). Therefore, the result immediately follows.\end{proof}

We provide lower bounds of the torsion in Chow group of codimension $2$ of the product of $n$ Severi-Brauer surfaces, which generalize the case of $n=2$ in \cite[Proposition 6.3]{IzhKar}.

\begin{proposition}\label{lowerSB}
	Let $n\geq 2$ be an integer and let $A_{1}, \ldots, A_{n} $ be central simple algebras such that $\ind(A_{1}^{\tens j_{1}}\tens \cdots \tens A_{n}^{\tens j_{n}})=3$ for any integers $0\leq j_{1},\ldots, j_{n}\leq 2$, not all equal to $0$. Then, the torsion subgroup $\CH^{2}(\bar{X})_{\tors}$ of the corresponding generic variety of $X=\prod_{i=1}^{n}\SB(A_{i})$ contains 
	\[(\Z/3\Z)^{\oplus N}, \text{ where } N=2^{n}+4{{n}\choose{3}}-(n+1).\]
	In particular, $\mathcal{M}(\SBS_{n})\geq 3^{N}$.
\end{proposition}
\begin{proof}
Let $A_{1}, \ldots, A_{n}$ be division algebras satisfying  $\ind(A_{1}^{\tens j_{1}}\tens \cdots \tens A_{n}^{\tens j_{n}})=3$ for any integers $0\leq j_{1},\ldots, j_{n}\leq 2$ and $X=\prod_{i=1}^{n}\SB(A_{i})$. Then, by (\ref{Quillenbasis}) we have a basis $\{1, 3x_{i_{1}}^{j_{1}}\cdots x_{i_{m}}^{j_{m}} \,| \newline 1\leq m\leq n, 1\leq i_{1}<\cdots< i_{m}\leq n, 0\leq j_{1},\ldots, j_{n}\leq 2 \}$ of $K(X)$, where for each $i\in I_{n}$, $x_{i}$ is the pullback of the class of the tautological line bundle on the projective plane.

Consider another basis $\{1, 3y_{i_{1}}^{j_{1}}\cdots y_{i_{m}}^{j_{m}}\}$ of $K(X)$, where $y_{i}=x_{i}-1$ for all $i\in I_{n}$. Let $B_{1}=\{b_{pq}:=3y_{p}^{2}y_{q}^{2}\}$, $B_{2}=\{b_{pqr}:=3y_{p}y_{q}y_{r}\}$, $B_{3}=\{b'_{pqr}:=3y_{p}^{2}y_{q}y_{r}\}$, where the indices $p, q, r$ range over all distinct integers in $I_{n}$. For each $3\leq m\leq n$, we set $D_{m}=\{d_{i_{1}\cdots i_{m}}:=3(y_{i_{1}}\cdots y_{i_{m}})^{2}\,|\, 1\leq i_{1}<\cdots< i_{m}\leq n\}$. Then, we get $|B_{1}|={{n}\choose{2}}$, $|B_{2}|={{n}\choose{3}}$, $|B_{3}|=3{{n}\choose{3}}$, $|D_{m}|={{n}\choose{m}}$, and $y_{i}^{3}=0$ for all $i\in I_{n}$. Set
\[B:=B_{1}\cup B_{2}\cup B_{3}\cup (\bigcup_{m=3}^{n}D_{m})\,\,\text{ and }\, N:=|B|=2^{n}+4{{n}\choose{3}}-(n+1).\]

It follows from (\ref{gammatwo}) that any element of $B$ is contained in $\Gamma^{2}(X)$. Assume that $b_{pq}\in\Gamma^{3}(X)$ for some $p, q$. Then, since $\Gamma^{3}(X)$ is generated by $\Gamma^{1}(X)\cdot \Gamma^{2}(X)$, $c_{3}(3x_{i_{1}}^{j_{1}}\cdots x_{i_{m}}^{j_{m}})$ and any element in $\Gamma^{1}(X)\cdot \Gamma^{2}(X)$ is divisible by $9$, we obtain 
\[3y_{p}^{2}y_{q}^{2}\equiv \sum_{1\leq i_{1}<\cdots< i_{m}\leq n,\, 0\leq j_{1},\ldots, j_{n}\leq 2} ( (y_{i_{1}}+1)^{j_{1}}\cdots (y_{i_{m}}+1)^{j_{m}}-1)^{3}\cdot \psi \,\,\mod 9\] for some $\psi\in K(X)$ depending on $i_{1},\ldots, i_{m}, j_{1},\ldots, j_{m}$. Therefore, by Lemma \ref{KeySB} together with substitution $y_{i}=-1$ for all $i\in I_{n}$ we obtain $3\equiv 0 \mod 9$. Hence, $b_{pq}\not\in \Gamma^{3}(X)$. By the same argument with Lemma \ref{KeySB}, we see that any element of $B$ is not contained in $\Gamma^{3}(X)$. As $3b_{pq}=3y_{p}^{2}(3y_{q}^{2})\in \Gamma^{4}(X)$, $3b_{pqr}=3y_{p}y_{q}(3y_{r})\in \Gamma^{3}(X)$, $3b'_{par}=3y_{p}^{2}(3y_{q}y_{r})\in \Gamma^{3}(X)$, $3d_{i_{1}\cdots i_{m}}\!\!=3y_{i_{1}}^{2}(3y_{i_{2}}^{2}\cdots y_{i_{m}}^{2})\in \Gamma^{4}(X)$, any element of $B$ gives a torsion of $\Gamma^{2/3}(X)$ of order $3$.

Now we show that any two subgroups generated by two different elements of $B$ have trivial intersection. Let $b_{pq}$ and $b_{rs}$ be two different elements in $B_{1}$. If $b_{pq}\pm b_{rs}\in \Gamma^{3}(X)$, then by the same argument as above together with substitution $y_{p}=y_{q}=-1, y_{i}=0$ for all $i\in I_{n}\backslash \{p, q\}$ we obtain $3\equiv 0 \mod 9$, which is a contradiction. Hence, the subgroups generated by $b_{pq}$ and $b_{rs}$ have trivial intersection. Moreover, by the same argument the subgroup generated by $b_{pq}$ has trivial intersection with any subgroup generated by an element of $B\backslash B_{1}$.

Let $z$ be either $1$ or a product of $x_{1},\cdots, x_{n}$ which does not contain any of $x_{p}$, $x_{q}$, and $x_{r}$. Consider the sequence $\beta'_{pqr}$ consisting of the coefficient of $b'_{pqr}/3$ in \begin{align}\label{sequenceofc3}
&c_{3}(3x_{p}^{2}x_{q}x_{r}z), c_{3}(3x_{p}x_{q}^{2}x_{r}z), c_{3}(3x_{p}x_{q}x_{r}^{2}z), c_{3}(3x_{p}^{2}x_{q}^{2}x_{r}z), c_{3}(3x_{p}^{2}x_{q}x_{r}^{2}), \\
&c_{3}(3x_{p}x_{q}^{2}x_{r}^{2}z),  c_{3}(3x_{p}^{2}x_{q}^{2}x_{r}^{2}z), c_{3}(3x_{p}x_{q}x_{r}z), \text{ respectively}.\nonumber
\end{align}
Then, by a direct calculation, we have $\beta'_{pqr}=(66, 30, 30, 132, 132, 60, 264, 15)$. Hence, each element of $\beta'_{pqr}-\beta'_{qpr}$, $\beta'_{pqr}-\beta'_{rpq}$, and $\beta'_{qpr}-\beta'_{rpq}$ is divisible by $9$, i.e., 
\begin{equation}\label{divisiblebetaprime}
9\mid \beta'_{pqr}-\beta'_{qpr},\, \beta'_{pqr}-\beta'_{rpq},\, \beta'_{qpr}-\beta'_{rpq}.
\end{equation}
Consider another sequence $\beta_{pqr}$ consisting of the coefficient of $b_{pqr}/3$ in (\ref{sequenceofc3}). Then, we get $\beta_{pqr}=(12,12,12,24,24,24,48, 6)$. Therefore, we have
\begin{equation}\label{divisiblebeta}
9\mid \beta_{pqr}-\beta'_{pqr},\, \beta_{pqr}-\beta'_{qpr},\, \beta_{pqr}-\beta'_{rpq}.
\end{equation}

Let $b'_{pqr}$ and $b'_{stu}$ be two different elements of $B_{3}$. If $b'_{pqr}\pm b'_{stu}\in \Gamma^{3}(X)$, then by applying Lemma \ref{KeySB} with $y_{p}=y_{q}=y_{r}=-1, y_{i}=0$ for all $i\in I_{n}\backslash \{p, q, r\}$ we get $3\equiv 0 \mod 9$, which is impossible. Therefore, the subgroups generated by $b'_{pqr}$ and $b'_{stu}$ have trivial intersection. Assume that $b'_{pqr}\pm b_{stu}\in \Gamma^{3}(X)$. Then, $b'_{pqr}-b'_{qpr}\in \Gamma^{3}(X)$, i.e., $b'_{pqr}-b'_{qpr}$ is a linear combination (over $\mathbb{Z}$) of elements in (\ref{sequenceofc3}) and the third Chern classes of basis elements of $K(X)$ which does not contain $x_{p}x_{q}x_{r}$. By comparing coefficients of $y_{p}y_{q}y_{r}$ we see that this contradicts the divisibility in (\ref{divisiblebetaprime}). By the same argument together with (\ref{divisiblebetaprime}), we get $b'_{pqr}\pm d_{i_{1}\cdots i_{m}}\not\in \Gamma^{3}(X)$. Hence, the subgroup generated by $b'_{pqr}$ has trivial intersection with any subgroup generated by an element in $B\backslash (B_{1}\cup B_{3})$.

Let $b_{pqr}$ and $b_{stu}$ be two different elements of $B_{2}$. Suppose that $b_{pqr}\pm b_{stu}\in \Gamma^{3}(X)$. Then, by applying Lemma \ref{KeySB} with $y_{p}=y_{q}=y_{r}=-1, y_{i}=0$ for all $i\in I_{n}\backslash \{p, q, r\}$ we obtain $-3\equiv 0 \mod 9$, which is a contradiction. Therefore, the subgroups generated by $b_{pqr}$ and $b_{stu}$ have trivial intersection. If $b_{pqr}\pm d_{i_{1}\cdots i_{m}}\in \Gamma^{3}(X)$, then by the same argument together with (\ref{divisiblebeta}) we obtain a contradiction. Therefore, the subgroup generated by $b_{pqr}$ has trivial intersection with any subgroup generated by an element in $B\backslash (B_{1}\cup B_{2}\cup B_{3})$.

For $3\leq m, k\leq n$, let $d_{i_{1}\cdots i_{m}}$ and $d_{i'_{1}\cdots i'_{k}}$ be two different elements of $D_{m}$ and $D_{k}$, respectively. If $d_{i_{1}\cdots i_{m}}\pm d_{i'_{1}\cdots i'_{k}}\in \Gamma^{3}(X)$, then by applying Lemma \ref{KeySB} with substitution $y_{i_{1}}=\cdots=y_{i_{m}}=-1$ and $y_{i}=0$ for all $i\in I_{n}\backslash \{i_{1}, \ldots, i_{m}\}$ we again get $3\equiv 0 \mod 9$, which is a contradiction. It follows that the group $\Gamma^{2/3}(X)_{\tors}$ contains $(\Z/3\Z)^{\oplus N}$, so does $\CH^{2}(\bar{X})_{\tors}$. \end{proof}


\subsection{Three Severi-Brauer surfaces} Let $A_{1}, A_{2}, A_{3}$ be central simple algebras of degree $3$ over $F$ and $X=$$\SB(A_{1})\times \SB(A_{2})\times \SB(A_{3})$. If one of $A_{1}, A_{2}, A_{3}$ is split, then by (\ref{reductionSB}) the problem to compute torsion in Chow group of codimension $2$ is reduced to the case of product of two Severi-Brauer varieties, which was done in \cite[Theorem 5.1]{IzhKar}. Hence, we may assume that $\ind(A_{i})=3$ for all $i\in I_{3}$. Let $e_{i}=\ind(A_j\tens A_k)$, $h_{i}=\ind(A_{j}^{\tens 2}\tens A_{k})$, $d=\ind(A_{1}\tens A_{2}\tens A_{3})$, and $g_{i}=\ind(A_{i}^{\tens 2}\tens A_{j}\tens A_{k})$ for all distinct $i, j, k\in I_{3}$. For the same reason, we may assume that $e_{i}, h_{i}, d, g_{i}\geq 3$. By (\ref{Quillenbasis}), we have a basis
\begin{equation}\label{basisthree}
\{1,\,3x_i,\,3x_{i}^{2},\,e_{i}x_{j}x_{k},\,h_{i}x_{j}^{2}x_{k},\,dx_{1}x_{2}x_{3},\,e_{i}x_{j}^{2}x_{k}^{2},\,g_{i}x_{i}^{2}x_{j}x_{k},\,g_{i}x_{i}x_{j}^{2}x_{k}^{2},\,dx_{1}^{2}x_{2}^{2}x_{3}^{2}\}
\end{equation}
of $K(X)$ of $27$ elements, where $x_{i}$ is the pullback of the class of the tautological line bundle on the projective plane. As before, we set $y_{i}=x_{i}-1$ for all $i\in I_{3}$.

We will need the following lemmas to find upper bounds of the torsion.
  
  
\begin{lemma}\label{Lem3}
Let $i, j, k$ be distinct integers in $I_{3}$. Then, we obtain

\smallskip

\noindent$(1)$ $3y_{j}y_{k}\in \Gamma^{2}(X)$, $3y_{j}^{2}y_{k}+3y_{j}y_{k}^{2}\in \Im(\res^{3/4})$ if $e_i=3$, 

\smallskip

\noindent$(2)$ $3y_{j}^{2}y_{k}-3y_{j}y_{k}^{2}\in \Im(\res^{3/4})$ if $h_{i}=3$,

\smallskip

\noindent$(3)$ $3\sum y_{p}y_{q}\in \Im(\res^{2/3})$, $6y_{1}y_{2}y_{3}+3\sum y_{p}^{2}y_{q}\in \Im(\res^{3/4})$, $9y_{1}^{2}y_{2}^{2}y_{3}^{2}\in \Gamma^{6}(X)$ if $d=3$,

\smallskip

\noindent$(4)$ $3(y_{j}y_{k}\!-y_{i}y_{k}\!-y_{i}y_{j})\!\!\in\! \Im(\res^{2/3})$, $3(y_{i}y_{j}^{2}+y_{i}y_{k}^{2})\!+\!3\sum y_{p}^{2}y_{q}+\!12y_{1}y_{2}y_{3}\!\in\! \Im(\res^{3/4})$ if $g_{i}=3$,

\smallskip

\noindent where the sums range over all $1\leq p, q\leq 3$.

\end{lemma}
\begin{proof}
Observe that $y_{1}^{3}=y_{2}^{3}=y_{3}^{3}=0$ in $K(X_{E})$. If $e_{i}=3$, then $3y_{j}y_{k}\in K(X)\cap \Gamma^{2}(X_{E})=\Gamma^{2}(X)$. As $c_{3}(3x_{j}x_{k})=3y_{j}^{2}y_{k}+3y_{j}y_{k}^{2}+6y_{j}^{2}y_{k}^{2}$, the result $(1)$ follows.
	
If $h_{i}=3$, then it follows from (\ref{basisthree}) and Lemma \ref{Lem3} ($1$) that $9y_{j}^{2}y_{k}=3y_{j}^{2}(3y_{k})$, $9y_{j}y_{k}^{2}=3y_{j}(3y_{k}^{2})\in \Gamma^{3}(X)$. Therefore, the result $(2)$ follows from $c_{3}(3x_{j}^{2}x_{k})=12y_{j}^{2}y_{k}+27y_{j}^{2}y_{k}^{2}+6y_{j}y_{k}^{2}$.

If $d=3$, then the result ($3$) for codimension $2$ immediately follows from $3(y_{1}+1)(y_{2}+1)(y_{3}+1)\in \Gamma^{2}(X)$. As $27(y_{1}y_{2}y_{3})^{2}\in \Gamma^{6}(X)$, the rest of them follow from the computations of $c_{3}(3x_{1}x_{2}x_{3})$ and $c_{6}(6x_{1}x_{2}x_{3})$.

If $g_{i}=3$, then by (\ref{basisthree}) we get $3x_{i}^{2}x_{j}x_{k}\in \Gamma^{2}(X)$, thus the result ($4$) for codimension $2$ follows from $3(y_{i}+1)^{2}(y_{j}+1)(y_{k}+1)\in \Gamma^{2}(X)$. Finally, the result for codimension $3$ immediately follows from $c_{3}(3x_{i}^{2}x_{j}x_{k})=12y_{i}^{2}(y_{j}+y_{k})+6y_{i}(y_{j}^{2}+y_{k}^{2})+3(y_{j}^{2}y_{k}+y_{j}y_{k}^{2})+12y_{1}y_{2}y_{3}$ modulo $\Gamma^{4}(X)$.\end{proof}

\begin{lemma}\label{Lem3prime}
Let $y_{i}=x_{i}-1$ for all $i\in I_{3}$. Then,

\smallskip

\noindent$(1)$ $3y_{1}^{2}y_{2}^{2}y_{3}-3y_{1}^{2}y_{2}y_{3}^{2}\in \Gamma^{3}(X)$ if $e_{i}=g_{i}=h_{i}=3$ for all $i\in I_{3}$.

\smallskip

\noindent$(2)$ $3y_{1}^{2}y_{2}^{2}y_{3}^{2}-6y_{1}^{2}y_{2}^{2}y_{3}\in \Gamma^{3}(X)$ if $e_{i}=g_{i}=h_{i}=d=3$ for all $i\in I_{3}$.
\end{lemma}
\begin{proof}
By direct calculation, one has
\begin{align*}
150(y_{1}^{2}y_{2}^{2}y_{3}\!-\!y_{1}^{2}y_{2}y_{3}^{2})\!=\!&-\!6(c_{3}(3x_{1}x_{2}^{2}x_{3})\!+c_{3}(3x_{1}x_{2}^{2}))\!+\!6(c_{3}(3x_{1}x_{2}x_{3}^{2})\!+c_{3}(3x_{1}x_{3}^{2}))\\
&\!+\!3(c_{3}(3x_{1}^{2}x_{2}^{2}x_{3})\!+c_{3}(3x_{2}^{2}x_{3}))\!-\!3(c_{3}(3x_{1}^{2}x_{2}x_{3}^{2})+c_{3}(3x_{2}x_{3}^{2}))\\
&-9c_{3}(3x_{1}^{2}x_{2})+9c_{3}(3x_{1}^{2}x_{3})+42c_{3}(3x_{1}x_{2})-42c_{3}(3x_{1}x_{3}).&
\end{align*}
Since $3y_{1}^{2}(3y_{2}^{2}y_{3})-3y_{1}^{2}(3y_{2}y_{3}^{2})\in \Gamma^{4}(X)$, the result $(1)$ follows. Similarly, we get \begin{align*}
150y_{1}^{2}y_{2}^{2}y_{3}^{2}\!-\!300y_{1}^{2}y_{2}^{2}y_{3}\!&=4(c_{3}(3x_{1}x_{2}x_{3}^{2})\!-\!c_{3}(3x_{1}^{2}x_{2}^{2}x_{3}))\!+\!6(c_{3}(3x_{1}^{2}x_{2})+c_{3}(3x_{1}x_{2}^{2}))\\
\!&\!+\!2(c_{3}(3x_{1}x_{3}^{2})+c_{3}(3x_{2}x_{3}^{2}))\!-\!2(c_{3}(3x_{1}^{2}x_{2}x_{3}^{2})+c_{3}(3x_{1}x_{2}^{2}x_{3}^{2}))\\
\!&\!+8(c_{3}(3x_{1}^{2}x_{2}x_{3})\!+c_{3}(3x_{1}x_{2}^{2}x_{3}))\!-\!8(c_{3}(3x_{1}x_{3})+c_{3}(3x_{2}x_{3}))\\
\!&\!+c_{3}(3x_{1}^{2}x_{2}^{2}x_{3}^{2})-16c_{3}(3x_{1}x_{2}x_{3})-36c_{3}(3x_{1}x_{2}).
\end{align*}
As $9y_{1}^{2}y_{2}^{2}y_{3}^{2}\in \Gamma^{4}(X)$ by Lemma \ref{Lem3} (3) and $3y_{1}^{2}(3y_{2}^{2}y_{3})\in \Gamma^{4}(X)$, the result $(2)$ follows.
\end{proof}

Applying Proposition \ref{lowerSB}, we prove the main result of this section. 


\begin{theorem}\label{threeproductSBThm}
The maximal torsion in Chow group of codimension $2$ of the product of three Severi-Brauer surfaces is $(\Z/3\Z)^{\oplus 8}$. In other words, $\mathcal{M}(\SBS_{3})=3^{8}$.
\end{theorem}
\begin{proof}
Let $A_{1}, A_{2}, A_{3}$ be division $F$-algebras of degree $3$ and $X=\SB(A_{1})\times \SB(A_{2})\times \SB(A_{3})$. Set $|K^{n}(X_{E})/K^{n}(X)|\cdot \beta_{n}=|\Gamma^{n/n+1}(X_{E})\Im(\res^{n/n+1})|$, where $E$ is a splitting field of $X$ and $K^{n}(X_{E})$ (resp. $K^{n}(X)$) is the codimension $n$ part of $K(X_{E})$ (resp. $K(X)$) for $1\leq n\leq 6$.

We start with some simple observations. First, by index calculation we have $h_{i}=3 \text{ or }9$ for all $i\in I_{3}$. Secondly, it follows from (\ref{basisthree}) and $c_{2}(3x_{i})=3y_{i}^{2}$ that $3y_{i}^{2}\in \Gamma^{2}(X)$ for all $i\in I_{3}$. Moreover, if $e_i=3$ for some $i\in I_{3}$, then by a simple calculation of index we see that the indexes $d, g_{1}, g_{2}, g_{3}$ are either $3$ or $9$ and $9y_{1}y_{2}y_{3}=3y_{j}y_{k}(3y_{i})\in \Gamma^{3}(X)$ by Lemma \ref{Lem3} (1).

Let $i, j, k$ be distinct integers in $I_{3}$, $h=h_{1}h_{2}h_{3}$, and $g=g_{1}g_{2}g_{3}$. Set \[G=|\{i\in I_{3}\,|\, g_{i}=3\}|\,\,\, \text{ and } H=|\{i\in I_{3}\,|\, h_{i}=9\}|.\] We shall find upper bounds of $\beta_n$ for $1\leq n\leq 6$. First of all, by (\ref{basisthree}) we have $\beta_{1}\leq 1$. For the rest of them, we will find upper bounds using case by case analysis.

\framebox{{\it Case}: $e_1=e_2=e_3=3$.} By Lemma \ref{Lem3} (1), we have $\beta_{2}\leq 1$. If $d=3$, then by Lemma \ref{Lem3} (1), (3) and $9y_{1}y_{2}y_{3}\in \Gamma^{3}(X)$ we get $3y_{1}y_{2}y_{3}\in \Im(\res^{3/4})$. Therefore, we may exclude a basis element $y_{1}y_{2}y_{3}$ for the computation of $\beta_{3}$ below. 

In codimension $3$, if $H=0$ (i.e., $h_{1}=h_{2}=h_{3}=3$), then it follows from Lemma \ref{Lem3} ($1$), ($3$) that $3y_{i}^{2}y_{j}\in \Im(\res^{3/4})$ for all $i, j\in I_{3}$, thus $\beta_{3}\leq 1$. If $H=1$ with $h_{k}=9$, then by Lemma \ref{Lem3} ($1$), ($2$) we have $3y_{i}^{2}y_{k}, 3y_{i}y_{k}^{2}, 3y_{j}^{2}y_{k}, 3y_{j}y_{k}^{2}\in \Im(\res^{3/4})$. For the remaining elements $y_{i}^{2}y_{j}, y_{i}y_{j}^{2}$ of the basis, we replace these by $y_{i}^{2}y_{j}+y_{i}y_{j}^{2}, y_{i}y_{j}^{2}+y_{i}y_{k}^{2}$ (resp. $y_{j}y_{k}^{2}+y_{j}y_{i}^{2}$) if $g_{i}=3$ (resp. $g_{j}=3$), thus it follows from Lemma \ref{Lem3} ($4$) that 
\[h^{2}\beta_{3}\leq 3^{4}\cdot 3^{2} \text{ (resp. }\leq 3^{4}\cdot 3\cdot 9) \text{ if } g_{i}=3 \text{ or } g_{j}=3  \text{ (resp. otherwise)}.\]
Similarly, by Lemma \ref{Lem3} ($4$) and an appropriate change of basis depending on $G$ we get
\begin{equation*}
h^{2}\beta_{3}\leq
\begin{cases}
3^{3}\cdot 3^{G}\cdot 9^{3-G} \text{ (resp. } 3^{4}\cdot 3^{G}\cdot 9^{2-G})& \text{ if } H=3 \text{ (resp. } H=2, h_{i}=3), G=0, 1,\\
3^{3}\cdot 3^{2}\cdot 9 \text{ (resp. } 3^{4}\cdot 3^{2}) & \text{ if } H=3 \text{ (resp. } H=2, h_{i}=3), G=2, 3.
\end{cases}
\end{equation*}

In the remaining cases, it follows from \ref{Lem3} ($1$) that $9y_{i}y_{j}y_{k}^{2}, 9y_{i}^{2}y_{j}^{2}\in \Im(\res^{4/5})$ for all $i, j, k$, thus we obtain $(3^{3}\cdot g)\beta_{4}\leq 9^{6}$. Again, by Lemma \ref{Lem3} ($1$), we have $3y_{i}^{2}(3y_{j}^{2}y_{k}+3j_{j}y_{k}^{2})\in \Im(\res^{4/5})$ for all $i, j, k$, thus $g\beta_{5}\leq 9^{3}$. It follows from Lemma \ref{Lem3} (3) that we get $\beta_{6}\leq 3$.

Combining all bounds $\beta_{n}$, we see that 
\begin{equation}\label{concase1}
|\!\oplus \Gamma^{n/n+1}(X)_{\tors}|\leq 3^{16}/g^{2} \text{ (resp. } \leq 3^{15}/g^{2}),\,\,\ \text{ if } H=0 \text{ (resp. otherwise.)}
\end{equation}
Therefore, the maximum upper bound of (\ref{concase1}) is $3^{10}$ when $g_{i}=h_{i}=3$ for all $i\in I_{3}$ and $d=3, 9$. If $g_{i}=h_{i}=d=3$, then it follows by Lemma \ref{Lem3} ($1$), ($2$) that $a:=3y_{1}^{2}y_{2}^{2}y_{3}-3y_{1}^{2}y_{2}y_{3}^{2}, b:=3y_{1}^{2}y_{2}^{2}y_{3}^{2}-6y_{1}^{2}y_{2}^{2}y_{3} \in \Gamma^{3}(X)\backslash \Gamma^{4}(X)$ as any element of $\Gamma^{4}(X)$ is divisible by $9$. Hence, the classes of $a$ and $b$ give torsion elements of $\Gamma^{3/4}(X)$ of order $3$ since $3a=3y_{1}^{2}(3y_{2}^{2}y_{3})-3y_{1}^{2}(3y_{2}y_{3}^{2}), 3b=9y_{1}^{2}y_{2}^{2}y_{3}^{2}-3y_{1}^{2}(3y_{2}^{2}y_{3})\in \Gamma^{4}(X)$. Moreover, we have $a-b$, $a+b\not\in \Gamma^{4}$ as any element of $\Gamma^{4}(X)$ is divisible by $9$, thus the subgroups generated by $a$ and $b$ have trivial intersection. By Proposition \ref{lowerSB}, we have
\[\Gamma^{2/3}(X)_{\tors}=(\Z/3\Z)^{\oplus 8} \text{ and } \Gamma^{3/4}(X)_{\tors}=(\Z/3\Z)^{\oplus 2}.\]
In this case, by (\ref{generic}) we have 
\begin{equation}\label{caseonecon}
\CH^{2}(\bar{X})_{\tors}=(\Z/3\Z)^{\oplus 8},
\end{equation}
where $\bar{X}$ is the corresponding generic variety. If $g_{i}=h_{i}=3$ for all $i\in I_{3}$ and $d=9$, then by the same argument the class $a$ gives a torsion of $\Gamma^{3/4}(X)$. Moreover, we have either $9y_{1}^{2}y_{2}^{2}y_{3}^{2}$ is a torsion of order $3$ in $\Gamma^{4/5}(X)\oplus \Gamma^{5/6}(X)$ or $\beta_{6}\leq 1$. Therefore, in any case we obtain
\begin{equation}\label{case1}
|\!\oplus \Gamma^{n/n+1}(X)_{\tors}|\leq 3^{8}.
\end{equation}


\framebox{{\it Case}: $e_i=e_j=3$ and $e_k=9$.} We first find upper bounds of $\beta_{n}$ for $n\neq 3$. If $d=3$ or $G\geq 1$ (say, $g_{k}=3$), respectively, then by Lemma \ref{Lem3} ($3$), ($4$) we can replace one of the elements in $K^{2}(X_{E})$ by $\sum_{1\leq p, q\leq 3} y_{p}y_{q}$ or $y_{ij}-y_{ik}-y_{jk}$. Therefore, $3^{7}\cdot \beta_{2}\leq 3^{5}\cdot 3$ if $d=3$ or $G\geq 1$. Otherwise, $3^{7}\cdot \beta_{2}\leq 3^{5}\cdot 9$. In codimension $4$, it follows by Lemma \ref{Lem3} ($1$) that $9y_{i}y_{j}y_{k}^{2}, 9y_{i}^{2}y_{j}^{2}\in \Im(\res^{4/5})$ for all $i, j, k\in I_{3}$. Hence, $(3^{4}g)\beta_{4}\leq 9^{6}$. In codimension $5$, by Lemma \ref{Lem3} ($1$), we have $3y_{i}^{2}(3y_{j}^{2}y_{k}+3y_{j}y_{k}^{2}), 3y_{j}^{2}(3y_{i}^{2}y_{k}+3y_{i}y_{k}^{2})\in \Im(\res^{5/6})$. If in addition $h_{i}=3$ (resp. $h_{j}=3$), then by Lemma \ref{Lem3} ($2$) we obtain $3y_{i}^{2}(3y_{j}^{2}y_{k}-3y_{j}y_{k}^{2})\in \Im(\res^{5/6})$ (resp. $3y_{j}^{2}(3y_{i}^{2}y_{k}-3y_{i}y_{k}^{2})\in \Im(\res^{5/6})$). Moreover, if $d=3$, then by Lemma \ref{Lem3} ($3$) $3y_{k}^{2}(-3y_{1}y_{2}y_{3}+3\sum_{1\leq p, q\leq 3} y_{p}^{2}y_{q})=9y_{i}^{2}y_{j}y_{k}^{2}+9y_{i}y_{j}^{2}y_{k}^{2}\in \Im(\res^{5/6})$. Therefore, we conclude that $g\beta_{5}\leq 9^{3}$ (resp. $\leq 9^{2}\cdot 27$) if one of $h_{i}, h_{j}, d$ is equal to $3$ (resp. otherwise). By Lemma \ref{Lem3} (3), we have $\beta_{6}\leq 3$. In conclusion, we get
\begin{equation}\label{case2prime}
g^{2}\beta_{2}\beta_{4}\beta_{5}\beta_{6}\leq 
\begin{cases}
3^{16} &\text{ if }G=0,\\
3^{15} &\text{ if }G\geq 1, d=h_{i}=h_{j}=9\\
3^{14} &\text{ otherwise.}
\end{cases}
\end{equation}

Now we calculate upper bounds for $\beta_{3}$ according to $H$. Since $9y_{1}y_{2}y_{3}\in \Gamma^{3}(X)$ and $3(y_{i}^{2}y_{j}+y_{i}y_{j}^{2})-3y_{1}y_{2}y_{3}\in \Im(\res^{3/4})$ if $d=3$, we may exclude $y_{1}y_{2}y_{3}$ for the computation of $\beta_{3}$ as in the previous case. If $H\geq 2$, then by Lemma \ref{Lem3} ($1$), ($2$), ($4$) we obtain $h^{2}\beta_{3}\leq 3^{2}\cdot 9\cdot g$ (resp. $\leq 3^{3}\cdot g$) if $H=3$ (resp. $H=2$) (we always use $y_{j}^{2}y_{k}+y_{j}y_{k}^{2}, y_{i}^{2}y_{k}+y_{i}y_{k}^{2}$ as a part of basis of $K^{3}(X_{E})$ and also use $y_{i}y_{j}^{2}+y_{i}y_{k}^{2}+\sum_{1\leq p, q\leq 3}y_{p}^{2}y_{q}$ if $g_{i}=3$). If $H=1$ (resp. $H=0, g_{k}=3$), then it follows from Lemma \ref{Lem3} ($1$), ($2$), ($4$) that $h^{2}\beta_{3}\leq 3^{4}\cdot 9^{2}$ (resp. $\leq 3^{5}\cdot 3$), i.e., $\beta_{3}\leq 1$. Finally, if $H=0$, $g_{k}=9$, then by Lemma \ref{Lem3} ($1$), ($2$) we have $h^{2}\beta_{3}\leq 3^{5}\cdot 9$. Hence, by (\ref{case2prime}) we see that for all cases
\begin{equation}\label{case2}
|\!\oplus \Gamma^{n/n+1}(X)_{\tors}|\leq 3^{8}.
\end{equation}


\framebox{{\it Case}: $e_i=3$ and $e_{j}=e_k=9$.} We slightly modify the previous argument to find upper bounds of $\beta_{n}$ for $n\neq 3$. In codimension $2$, by the same argument as in the previous case we have $3^{8}\beta_{2}\leq 3^{4}\cdot \min\{g_{j}, g_{k}\}\cdot \min\{d, g_{i}\}$. In codimension $4$, it follows from Lemma \ref{Lem3} ($1$) that $9y_{i}^{2}y_{j}y_{k}, 3y_{i}\cdot 3(y_{j}^{2}y_{k}+y_{j}y_{k}^{2})\in \Im(\res^{4/5})$. Hence, by Lemma \ref{Lem3} ($2$) we obtain $(3^{5}g)\beta_{4}\leq 9^{6}$ (resp. $\leq 9^{5}\cdot 27$) if $H\neq 3$ (resp. $H=3$). In codimension $5$ and $6$, by the same argument as in the previous case we get
\[g\beta_{5}\leq 9^{3} \text{ (resp. }\leq 9^{4-H}\cdot 27^{H-1} )\,\, \text{ if } d=3 \text{ or } H\leq 1\text{ (resp. otherwise).}\]
and $\beta_{6}\leq 3$. Hence, 
\begin{equation}\label{case3prime}
g^{2}\beta_{2}\beta_{4}\beta_{5}\beta_{6}\leq 
\begin{cases}
3^{10}\cdot \min\{d, g_{i}\}\cdot \min\{g_{j}, g_{k}\} &\text{ if }d=3 \text{ or } H\leq 1,\\
3^{11}\cdot  g_{i}\cdot \min\{g_{j}, g_{k}\} &\text{ if }d=9, H=2,\\
3^{13}\cdot  g_{i}\cdot \min\{g_{j}, g_{k}\} &\text{ if }d=9, H=3.
\end{cases}
\end{equation}

In codimension $3$, by the same argument as in the previous case together with Lemma \ref{Lem3} ($1$), ($2$), ($4$) we have 
\begin{equation}\label{case3primeprime}
h^{2}\beta_{3}\leq 
\begin{cases}
3\cdot 9^{2}\cdot g  \text{ (resp. } 3^{2}\cdot 9\cdot g) &\text{ if }H=3 \text{ (resp. } H=2),\\
3^{3}\cdot 9^{3} \text{ (resp. } 3^{4}\cdot 9^{2}) &\text{ if }H=1 \text{ (resp. } H=0).
\end{cases}
\end{equation}
Therefore, by (\ref{case3prime}) and (\ref{case3primeprime}) we see that for all cases 
\begin{equation}\label{case3}
|\!\oplus \Gamma^{n/n+1}(X)_{\tors}|\leq 3^{8}.
\end{equation}


\framebox{{\it Case}: $e_1=e_2=e_3=9$.} In codimension $2$, we replace one of the basis elements in $K^{2}(X_{E})$ by $y_{j}y_{k}-y_{i}y_{k}-y_{i}y_{j}$ (resp. $\sum_{1\leq p, q\leq 3}y_{p}y_{q}$) if $g_{i}=3$ (resp. $d=3$). Then, by Lemma \ref{Lem3} ($3$), ($4$) we obtain
\[(3^{3}\cdot 9^{3})\beta_{2}\leq 3^{6} \text{ (resp. }\leq 3^{3+G+[3/d]}\cdot 9^{3-(G+[3/d])} )\,\, \text{ if } G=d=3 \text{ (resp. otherwise).}\]
In codimension $3$, we use $y_{j}^{2}y_{k}-y_{j}y_{k}^{2}$ (resp. $y_{i}y_{j}^{2}+y_{i}y_{k}^{2}+\sum_{1\leq p, q\leq 3}y_{p}^{2}y_{q}+y_{1}y_{2}y_{3}$) as a basis element in $K^{3}(X_{E})$ if $h_{i}=3$ (resp. $g_{i}=3$). Moreover, if $d=3$, then we also use $-y_{1}y_{2}y_{3}+\sum_{1\leq p, q\leq 3}y_{p}^{2}y_{q}$ as a basis element in $K^{3}(X_{E})$. Then, as $9y_{1}y_{2}y_{3}\in \Gamma^{3}(X)$, it follows from Lemma \ref{Lem3} ($2$), ($3$), ($4$) that 
\[h^{2}\beta_{3}\leq 3^{3-H+G}\cdot 9^{3+H-G}\, (\text{resp. }\leq 3^{2-H+G}\cdot 9^{3+H-G}) \text{ if }d=3, 9 \,(\text{resp. } d=27). \]
In codimension $4$, we replace the element $y_{i}y_{j}y_{k}^{2}$ by $y_{i}(y_{j}^{2}y_{k}-y_{j}y_{k}^{2})$ if $h_{i}=3$. Since $9y_{i}^{2}y_{j}^{2}, 27y_{i}^{2}y_{j}y_{k}\in \Gamma^{4}(X)$ for all $i, j, k\in I_{3}$, it follows from Lemma \ref{Lem3} ($2$) that
\[(9^{3}\cdot g)\beta_{4}\leq 9^{3}\cdot 9^{3-H}\cdot 27^{H} \text{ (resp. }\leq 9^{5}\cdot 27)\,\, \text{ if } H\neq 0 \text{ (resp. otherwise).}\]

In codimension $5$, if $d=3$, then it follows by Lemma \ref{Lem3} (3) that $3y_{i}^{2}\cdot (6y_{1}y_{2}y_{3}+3\sum_{1\leq p, q\leq 3}y_{p}^{2}y_{q})=9y_{i}^{2}y_{j}^{2}y_{k}\in \Im(\res^{5/6})$ for all $i, j, j\in I_{3}$. If $d\neq 3$ and $h_{i}=3$, then we replace the element $y_{i}^{2}y_{j}^{2}y_{k}$ by $y_{i}^{2}(y_{j}^{2}y_{k}-y_{j}y_{k}^{2})$. Hence, by Lemma \ref{Lem3} ($2$) we get
\begin{equation*}
g\beta_{5}\leq
\begin{cases}
9^{3} &\text{ if }d=3,\\
9^{2}\cdot 27 (\text{ resp. } 9^{3-H}\cdot 27^{H}) &\text{ if }d\neq 3, H=0 \text{ (resp. }H\neq 0).
\end{cases}
\end{equation*} 
Since $27y_{1}^{2}y_{2}^{2}y_{3}^{2}\in \Gamma^{6}(X)$ and $9y_{1}^{2}y_{2}^{2}y_{3}^{2}\in \Gamma^{6}(X)$ if $d=3$, we obtain $\beta_{6}\leq 1$ (resp. $\leq 3$) if $d=27$ (resp. otherwise). Combining all bounds of $\beta_{n}$, we see that for all cases 
\begin{equation}\label{case4}
|\!\oplus \Gamma^{n/n+1}(X)_{\tors}|\leq 3^{6}.
\end{equation}
In conclusion, the result follows from (\ref{caseonecon}), (\ref{case1}), (\ref{case2}), (\ref{case3}), and (\ref{case4}).\end{proof}

\section{Product of Quadric surfaces}\label{section6}


In this section, we obtain upper bounds for the torsion in Chow group of codimension $2$ of the product of two quadric surfaces (Theorem \ref{twoquadricsurface}) and the product of three quadric surfaces with the same discriminant (Proposition \ref{threequadricsurface}). In the case of the product of two quadric surfaces, we provide a sharp lower bound in the gamma filtration (Proposition \ref{lowerboundgammaintwoquadric}).

Let $F$ be a field of characteristic different from $2$ and let $q=\langle c, -a, -b, ab\rangle$ be a nondegenerate quadratic form over $F$  for $a, b, c\in F^{\times}$. If the discriminant is trivial, then the quadric surface corresponding to $q$ is birational to $\P^{1}\times \SB(Q)$, where $Q=(a, b)$ is a  quaternion $F$-algebra. Otherwise, the quadric is isomorphic to $R_{L/F}(\SB(Q))$, where  $R_{L/F}$ is the Weil restriction over a quadratic field $L=F(\sqrt{c})$. We write $\disc Q$ for the discriminant $c$.

Consider two quadric surfaces with the corresponding quaternions $Q_{1}$, $Q_{2}$ and quadratic extensions $L_{1}$, $L_{2}$ as above. We set
\begin{equation}\label{associatedvariety}
X=
\begin{cases}
\SB(Q_{1})\times \SB(Q_{2}) & \text{ if } \disc Q_i=1,\\
\SB(Q_{1})\times R_{L_{2}/F}(\SB(Q_{2})) & \text{ if } \disc Q_{1}=1\neq \disc Q_{2},\\
R_{L_{1}/F}(\SB(Q_{1}))\times R_{L_{2}/F}(\SB(Q_{2})) & \text{ if } \disc Q_{i}\neq 1
\end{cases}
\end{equation}
for all $i$. Then, by \cite[Corollary 2.5]{IzhKar98} the group $\CH^{2}(X)_{\tors}$ of the first and second cases of (\ref{associatedvariety}) is isomorphic to that of the product of two quadric surfaces. Hence, for torsion it suffices to consider $X$, called the variety associated to the product of two quadric surfaces. 

Consider the last case of (\ref{associatedvariety}). If $\ind(Q_{1})_{L_{1}}=\ind(Q_{2})_{L_{2}}=1$, then the associated variety $X$ has torsion-free Chow groups. Thus, we may assume that $\ind(Q_{1})_{L_{1}}=2$. We choose a splitting field $E$ of X as follows. If $\ind(Q_{2})_{L_{2}}=1$, then we take a maximal subfield ($\neq L_{2}$) of $Q_{1}$ for $E$. Otherwise, we take for $E$ a common maximal subfield ($\neq L_{1}, L_{2}$) of $Q_{1}$ and $Q_{2}$ if $\ind(Q_{1}\tens Q_{2})\leq 2$ or the tensor product of maximal subfields $E_{1}$($\neq L_{2}$) of $Q_{1}$ and $E_{2}$($\neq L_{1}$) of $Q_{2}$ if $\ind(Q_{1}\tens Q_{2})=4$. Hence, $d:=[E:F]=4$ if $\ind(Q_{1}\tens Q_{2})=4$ and $d=2$ otherwise. For the second case (\ref{associatedvariety}), we choose a splitting field $E$ in the same way.

The theorem below was proven in \cite{IzhKar98}. Here, we give an elementary short proof which does not use any cohomological method or $K$ theory of quadrics. Moreover, we find upper bound of the total torsion in the topological filtration of the product of two quadric surfaces with nontrivial discriminants.

\begin{theorem}\cite[Theorems 5.1, 5.7, 5.8, 5.9]{IzhKar98}\label{twoquadricsurface}
Let $X$ be the variety associated to the product of two quadric surfaces. Then, the group $\CH^{2}(X)_{\tors}$ is either $0$ or $\Z/2\Z$, i.e., $\mathcal{M}(\QS_{2})\leq 2$.
\end{theorem}
\begin{proof}
Let $Q_{i}$ be a quaternion algebra with a quadratic extension $L_{i}$ for $i=1, 2$, $X$ the variety associated to the product of two quadric surfaces, and E the splitting field of $X$ as above. If $\disc Q_{i}=1$ for all $i$, then the variety $X$ is torsion free. In the remaining cases, we find upper bounds of $\alpha_{n}:=|T^{n/n+1}(X_{E})/\Im(\res^{n/n+1})|$ for each of the following $3$ cases.

\framebox{{\it Case}: $L_{1}L_{2}:=L_{1}\tens L_{2}$ is a biquadratic field extension.} Let $L=L_{1}L_{2}$. Then, we have $X_{E}=R_{EL_{1}/E}(\P^{1})\times R_{EL_{2}/E}(\P^{1})$ and $X_{EL}=(\P^{1})^{4}$, which have torsion-free Chow groups. Let $d_{i}=\ind(Q_{i})_{L}$ and $e=\ind(Q_{1}\tens Q_{2})_{L}$. Then, by the action of the Galois group of $EL/E$, we have the following bases of $K(X_{E})$ and $K(X)$, respectively:
\begin{equation}\label{case1KXE}
\{1, x_{2i-1}+x_{2i}, x_{12}, x_{34}, (x_1+x_2)(x_3+x_4), x_{12}(x_{3}+x_{4}), x_{34}(x_{1}+x_{2}), x_{1234}\, |\, i=1,2\} \text{ and}
\end{equation}
\begin{equation}\label{case1KX}
\{1, d_{i}(x_{2i-1}+x_{2i}), x_{12}, x_{34}, e(x_1+x_2)(x_3+x_4), d_{2}x_{12}(x_{3}+x_{4}), d_{1}x_{34}(x_{1}+x_{2}), x_{1234}\},
\end{equation}
where $x_{i}$ is the pullback of the class of tautological line bundle on the projective space in $K(X_{EL})$ for all $i\in I_{4}$ and $x_{i_{1}\cdots i_{m}}=x_{i_{1}}\cdots x_{i_{m}}$ for $1\leq i_{1}<\ldots< i_{m}\leq 4$. Hence, $|K(X_{E})/K(X)|=d_{1}^{2}d_{2}^{2}e$. If $d_{1}=d_{2}=1$, then $e=1$, thus we may assume that $d_{1}d_{2}\geq 2$.

Let $y_{i}=x_{i}-1$ and $y_{i_{1}\cdots i_{m}}=y_{i_{1}}\cdots y_{i_{m}}$ for $1\leq i_{1}<\ldots< i_{m}\leq 4$. We will use other bases for $K(X_{E})$ the basis (\ref{case1KXE}) by replacing $x_{i}$ by $y_{i}$ and for $K(X)$
\begin{equation}\label{case1KXY}
\{1, d_{i}(y_{2i-1}+y_{2i}), z_{12}, z_{34}, e(y_{1}+y_{2})(y_{3}+y_{4}), d_{2}z_{12}(y_{3}+y_{4}), d_{1}z_{34}(y_{1}+y_{2}), z_{12}z_{34}\},
\end{equation} 
where $z_{12}=y_{1}+y_{2}+y_{12}$ and $z_{34}=y_{3}+y_{4}+y_{34}$. As $z_{12}, z_{34}\in K(X)$, we have $y_{1}+y_{2}, y_{3}+y_{4}\in \Im(\res^{1/2})$, thus $\alpha_{1}=1$ and $(y_{1}+y_{2})(y_{3}+y_{4})\in \Im(\res^{2/3})$. In addition, it follows from the basis and (\ref{gammatwo}) that $d_{1}y_{12}, d_{2}y_{34}\in T^{2}(X)$. Hence, $\alpha_{2}\leq d_{1}d_{2}$.

If $\ind(Q_{1}\tens Q_{2})\!=1$, then by the closed embeddings $\SB(Q_{1})\times R_{L_{2}/F}(\SB(Q_{1}))\hookrightarrow X$ and $R_{L_{1}/F}(\SB(Q_{1}))\times \SB(Q_{1})\hookrightarrow X$ we have $y_{34}(y_{1}+y_{2}), y_{12}(y_{3}+y_{4})\in \Im(\res^{3/4})$, respectively. Otherwise, it follows from $d_{1}y_{12}\cdot (y_{34}+y_{3}+y_{4})$ and $d_{2}y_{34}\cdot (y_{12}+y_{1}+y_{2})$ that $d_{1}y_{12}(y_{3}+y_{4}), d_{2}y_{34}(y_{1}+y_{2})\in \Im(\res^{3/4})$. Hence, $\alpha_{3}\leq 1$ (resp. $\leq d_{1}d_{2}$) if $\ind(Q_{1}\tens Q_{2})=1$ (resp. otherwise). By a transfer argument, we have $\alpha_{4}\leq d$. Hence, we obtain
$$|\!\oplus T^{n/n+1}(X)_{\tors}|= 1 \text{ (resp. } \leq d/e) \text{ if } \ind(Q_{1}\tens Q_{2})=1 \text{ (resp. otherwise)}.$$ If $e=1$, then $\ind(Q_{1}\tens Q_{2})\leq 2$, thus $d=2$ and $|T^{2/3}(X)_{\tors}|\leq 2$. Moreover, the group $\oplus T^{n/n+1}(X)_{\tors}$ is trivial if $e=4$ or $e=d=2$ or $\ind(Q_{1}\tens Q_{2})=1$.


\framebox{{\it Case}: $L_{1}=L_{2}$.} Let $L=L_{1}=L_{2}$. Then, $X_{E}=R_{EL/E}(\P^{1})\times R_{EL/E}(\P^{1})$. Applying the same argument as in the previous case, we have the basis (\ref{case1KXE}) (resp. (\ref{case1KX})) replacing $(x_{1}+x_{2})(x_{3}+x_{4})$ (resp. $e(x_{1}+x_{2})(x_{3}+x_{4})$) with two elements $x_{13}+x_{24}$ and $x_{14}+x_{23}$ (resp. $e(x_{13}+x_{24})$ and $e(x_{14}+x_{23})$) for $K(X_{E})$ (resp. $K(X)$). As $d_{1}=2$, $|K(X_{E})/K(X)|=4d_{2}^{2}e^2$. Similarly, we use other bases for $K(X_{E})$ the basis by replacing $x_{i}$ with $y_{i}$ and for $K(X)$ the basis (\ref{case1KXY}) replacing $e(y_{1}+y_{2})(y_{3}+y_{4})$ with $e(y_{13}+y_{24}+\sum_{i=1}^{4} y_{i})$ and $e(y_{14}+y_{23}+\sum_{i=1}^{4} y_{i})$.

Repeating the same argument as before yields $\alpha_{1}=1$. In codimension $2$, we have $2y_{12}, d_{2}y_{34}, (y_{1}+y_{2})(y_{3}+y_{4})\in \Im(\res^{2/3})$. If $e=1$, then $d_{2}=d=2$, and $X=R_{L/F}(\SB(Q_{1}))\times R_{L/F}(\SB(Q_{1}))$, thus by the diagonal embedding $R_{L/F}(\SB(Q_{1})) \hookrightarrow X$ the sum of all basis elements in $K^{2}(X_{E})$ is contained in $\Im(\res^{2/3})$. Moreover, if $e\neq 1$, then $e(y_{13}+y_{24})$, $e(y_{14}+y_{23})\in T^{2}(X)$. Therefore, we obtain $\alpha_{2}\leq 2d_{2}e$ in both cases.

If $e=1$, then $z_{12}(y_{12}+y_{34}+y_{13}+y_{24}+y_{14}+y_{23}), 2y_{12}(y_{3}+y_{4})\in \Im(\res^{3/4})$. Otherwise, we obtain $2y_{12}(y_{3}+y_{4}), d_{2}y_{34}(y_{1}+y_{2})\in  \Im(\res^{3/4})$. Hence, we have $\alpha_{3}\leq 2$ (resp. $\leq 2d_{2}$) if $e=1$ (resp. otherwise). Finally, we get $\alpha_{4}\leq d$, thus \begin{equation}\label{case2torsioneq}
|\!\oplus T^{n/n+1}(X)_{\tors}|= 1 \text{ (resp. } d/e ) \text{ if } e=1 \text{ (resp. otherwise)}.
\end{equation} 
Hence, the group $\oplus T^{n/n+1}(X)_{\tors}$ is trivial except the case where $e=2$ and $d=4$. In this case, we can further reduce the upper bound of $\alpha_{4}$ to $2$ if $2y_{1234}\in T^{4}(X)$. Hence, the group $\oplus T^{n/n+1}(X)_{\tors}$ is trivial in this case. If $2y_{1234}\notin T^{4}(X)$, then the class of $2y_{1234}=2y_{12}\cdot z_{34}-z_{12}\cdot 2(y_{13}+y_{24})\in T^{3}(X)$ gives a torsion element of order $2$ in $T^{3/4}(X)_{\tors}$. Therefore, the group $T^{2/3}(X)_{\tors}$ is trivial in all cases.

\framebox{{\it Case}: $\disc Q_{1}=1$.} Then, we have $X_{E}=\P^{1}\times R_{L_{2}E/E}(\P^{1})$ and $X_{EL_{2}}=(\P^{1})^{3}$, thus we obtain the following bases of $K(X_{E})$ and $K(X)$, respectively:
\[\{1, x_{1}, x_{2}+x_{3}, x_{23}, x_{1}(x_{2}+x_{3}), x_{123}\} \text{ and } \{1, d_{1}x_{1}, d_{2}(x_{2}+x_{3}), x_{23}, ex_{1}(x_{2}+x_{3}), d_{1}x_{123}\},\]
where $d_{1}=\ind(Q_{1})_{L_{2}}, d_{2}=\ind(Q_{2})_{L_{2}}, e=\ind(Q_{1}\tens Q_{2})_{L_{2}}$, and $x_{i}$ is the pullback of the class of tautological line bundle on the projective space in $K(X_{EL_{2}})$ for $i\in I_{3}$. Hence, $|K(X_{E})/K(X)|=d_{1}^{2}d_{2}e$. We use other bases for $K(X_{E})$ the above basis by replacing $x_{i}$ by $y_{i}$ and for $K(X)$ 
\begin{equation}\label{case3KXYBASIS}
\{1, d_{1}y_{1}, d_{2}(y_{2}+y_{3}), z_{23}, e(y_{1}+1)(y_{2}+y_{3}), d_{1}y_{1}z_{23}\}, \text{ where } z_{23}=y_{2}+y_{3}+y_{23}.
\end{equation}

Obviously, $\alpha_{1}\leq d_{1}$. In codimension $2$, we have $d_{2}y_{23}, d_{1}y_{1}(y_{2}+y_{3})\in T^{2}(X)$. If $e=1$, then $(y_{12}+y_{13}+y_{2}+y_{3}) +z_{23}-2(y_{2}+y_{3})=y_{12}+y_{13}+y_{23}\in T^{2}(X)$. Hence, we get $\alpha_{2}\leq \min\{d_{1}, d_{2}\}$ (resp. $\leq d_{1}d_{2}$) if $e=1$ (resp. otherwise). Finally, we have $\alpha_{3}\leq d$, thus the same upper bound (\ref{case2torsioneq}) is obtained for $|\!\oplus T^{n/n+1}(X)_{\tors}|$. Therefore, the group $\oplus T^{n/n+1}(X)_{\tors}$ is trivial if $e=1$ or $e=d=2$ or $e=4$. Assume that $e=2$ and $d=4$. Then it follows from (\ref{case3KXYBASIS}) that $2y_{123}\in T^{2}(X)$. If $2y_{123}\in T^{3}(X)$, then $\alpha_{3}\leq 2$. Hence, the group $\oplus T^{n/n+1}(X)_{\tors}$ is trivial in this case. Otherwise, the class of $2y_{123}$ gives a torsion of order $2$ in $T^{2/3}(X)$. In any case, $|T^{2/3}(X)_{\tors}|\leq 2$, thus the result follows from (\ref{tolchow}).
\end{proof}
\begin{remark}\label{remarkforproductofquadric}
Note that the proof of Theorem \ref{twoquadricsurface} indicates when the group $\CH^{2}(X)_{\tors}$ is trivial. Moreover, the proof still works if we replace the topological filtration by the gamma filtration.
\end{remark}

Now we provide a nontrivial torsion subgroup in the gamma filtration:

\begin{proposition}\label{lowerboundgammaintwoquadric}
With the above notations, we have $\Gamma^{2/3}(X)_{\tors}=\Z/2\Z$ if $L$ is a biquadratic extension, $e=2$, and $d=4$.
\end{proposition}
\begin{proof}
It follows from the basis (\ref{case1KXY}) that $2y_{1234}\in K(X)$. Hence, by (\ref{gammatwo}) we have $2y_{1234}\in \Gamma^{2}(X)$. As $d=4$, $4y_{1234}\in \Gamma^{3}(X)$. We show that the element $2y_{1234}$ is not contained in $\Gamma^{3}(X)$ by computing the Chern classes of the elements in the basis (\ref{case1KX}). Consider the basis (\ref{case1KX}) of $K(X)$ with $e=d_{1}=d_{2}=2$. Then, it follows from Whitney formula that we have $c_{1}(2(x_{2i-1}+x_{2i}))=2(y_{2i-1}+y_{2i})$, $c_{1}(x_{2i-1}+x_{2i})=y_{2i-1}+y_{2i}$, and $c_{2}(x_{2i-1}+x_{2i})=y_{2i-1}y_{2i}$ for all $i=1, 2$. Therefore, we have 
\begin{equation}\label{codimensiononebasis}
c_{2}(2(x_{2i-1}+x_{2i}))=4y_{2i-1}y_{2i},\, c_{1}(x_{2i-1}x_{2i})^{2}=2y_{2i-1}y_{2i}\, \text{ and }\, c_{1}(x_{2i-1}x_{2i})^{3}=0.
\end{equation}
Similarly, we obtain $c_{j}(2(x_{2i-1}+x_{2i}))=0$ for $j=3, 4$.

Let $z=(x_{1}+x_{2})(x_{3}+x_{4})$, $z'=(y_{1}+y_{2})(y_{3}+y_{4})$, $u'=y_{123}+y_{124}$, and $v'=y_{134}+y_{234}$. Then, by a direct computation, we have 
\begin{equation}\label{codimensiontwobasis}
c_{j}(z)=
\begin{cases}
2(\sum_{i=1}^{4}y_{i})+z' & \text{ for } j=1,\\
2y_{1234}+3z'+4(u'+v'+y_{12}+y_{34}) & \text{ for } j=2,\\
4(u'+v'+3y_{1234}) & \text{ for } j=3,\\
2y_{1234} & \text{ for } j=4.
\end{cases}
\end{equation}
Since $c_{2}(2z)=c_{1}(z)^{2}+2c_{2}(z)$, it follows from (\ref{codimensiontwobasis}) that
\begin{equation}\label{codimensiontwobasisprime}
c_{2}(2z)=8y_{1234}+14z'+16(u'+v'+y_{12}+y_{34}).
\end{equation}
As $c_{3}(2z)=2(c_{1}(z)c_{2}(z)+c_{3}(z))$ and $c_{4}(2z)=2(c_{1}(z)c_{3}(z)+c_{4}(z))+c_{2}(z)^{2}$, by (\ref{codimensiontwobasis})
\begin{equation}\label{codimensiontwobasisprimetwo}
c_{j}(2z)\equiv 0 \mod 4 \text{ for } j=3, 4.
\end{equation}

Let $u=x_{12}(x_{3}+x_{4})$. Then, we obtain $c_{1}(u)=2z_{12}+y_{3}+y_{4}+z'+u'$ and $c_{2}(u)$$=4y_{1234}+3u'+2(y_{12}+v')+y_{13}+y_{14}+y_{23}+y_{24}+y_{34}$. Therefore, we have
\begin{equation}\label{codimensionthreebasis}
c_{j}(2u)\!=\!
\begin{cases}
2(8y_{1234}\!+9u'\!+4v'+6y_{12}+3y_{13}+3y_{23}+3y_{14}+3y_{24}+2y_{34}) &\!\text{ for } j=2,\\
4(10y_{1234}+3u'+2v') &\!\text{ for } j=3,\\
8y_{1234} &\!\text{ for } j=4.
\end{cases}
\end{equation}
Let $v=x_{34}(x_{1}+x_{2})$. Then, we have the Chern classes (\ref{codimensionthreebasis}) for $2v$ by replacing indexes $1$, $2$, $3$, $4$ with $3$, $4$, $1$, $2$, respectively.

It follows from (\ref{codimensiononebasis}) that $c_{1}(x_{12})^{2}c_{1}(x_{34})=c_{1}(x_{12})^{2}c_{1}(x_{1234})=2(u'+y_{1234})$, $c_{1}(x_{34})^{2}c_{1}(x_{12})\\=c_{1}(x_{34})^{2}c_{1}(x_{1234})=2(v'+y_{1234})$. As the first Chern classes of $2z, 2u, 2v, 2(x_{2i-1}+x_{2i})$ are divisible by $4$, one easily sees that the subgroup generated by the products of three of the first Chern classes of any element in the basis (\ref{case1KX}) is generated by $2(u'+y_{1234})$ and $2(v'+y_{1234})$ modulo $4$. Similarly, using (\ref{codimensiononebasis}), (\ref{codimensiontwobasisprime}), (\ref{codimensionthreebasis}) one sees that the subgroup generated by the products of the first and second Chern classes of any element in the basis (\ref{case1KX}) is also generated by $2(u'+y_{1234})$ and $2(v'+y_{1234})$ modulo $4$. It follows from (\ref{codimensiontwobasisprimetwo}) and (\ref{codimensionthreebasis}) that the third and fourth Chern classes of $2z, 2u, 2v$ are divisible by $4$, thus, the group $\Gamma^{3}(X)$ is generated by $2(u'+y_{1234})$ and $2(v'+y_{1234})$ modulo $4$. Hence, $2y_{1234}\notin \Gamma^{3}(X)$ and this element gives a torsion of $\Gamma^{2/3}(X)$ of order $2$, thus the result follows from Remark \ref{remarkforproductofquadric}. \end{proof}

\begin{remark}
If $\bar{X}$ is a corresponding generic variety to $X$ in Proposition \ref{lowerboundgammaintwoquadric}, then we obtain $\CH^{2}(\bar{X})=\Z/2\Z$, which recovers \cite[Theorem 14.1]{IzhKar2000}. Indeed, it is possible to find such a variety by showing that the gamma filtration for $R_{L_{1}/F}(\SB(Q'_{1}))\times R_{L_{2}/F}(\SB(Q'_{2}))\times R_{L/F}(\SB(2, Q'_{1}\tens Q'_{2}))$ is torsion-free, where $L=L_{1}L_{2}$ is a biquadratic extension and $\SB(2, Q'_{1}\tens Q'_{2})$ is the generalized Severi-Brauer variety of rank $2$ left ideals in $(Q'_{1}\tens Q'_{2})_{L}$.
\end{remark}

\subsection{Three quadric surfaces with the same discriminant.}
Let $Q_1$, $Q_{2}$, $Q_{3}$ be three quaternion $F$-algebras and let $L$ be the quadratic extension over $F$ corresponding to three quadratic surfaces with the same discriminant. We set 
\begin{equation}\label{threequaternionassociated}
X=
\begin{cases}
\SB(Q_{1})\times \SB(Q_{2})\times \SB(Q_{3}) & \text{ if } \disc Q_i=1,\\
R_{L/F}(\SB(Q_{1}))\times R_{L/F}(\SB(Q_{2}))\times R_{L/F}(\SB(Q_{3}))& \text{ otherwise} 
\end{cases}
\end{equation}
Then, by the same argument as in the case of two quadric surfaces, the group $\CH^{2}(X)_{\tors}$ is isomorphic to that of the product of three quadric surfaces with the same discriminant.

Consider the second case of (\ref{threequaternionassociated}). If $\ind(Q_i)_{L}=1$ for all $i\in I_{3}$, then $X$ has torsion-free Chow groups, thus we may assume that $\ind(Q_{1})_{L}=2$. We choose a splitting field $E$ of $X$ as follows. If $\ind(Q_{2})_{L}=\ind(Q_{3})_{L}=1$, then we take a maximal subfield of $(Q_{1})_{L}$ for $E$. Assume $\ind(Q_{2})_{L}=2$, $\ind(Q_{3})_{L}=1$. Then, we take for $E$ a common maximal subfield of $(Q_{1})_{L}$ and $(Q_{2})_{L}$ (resp. the tensor product of maximal subfields of $(Q_{1})_{L}$ and $(Q_{2})_{L}$) if $\ind(Q_{1}\tens Q_{2})_{L}\leq 2$ (resp. otherwise). 

Let $G_{n}=|\{ij\in \{12, 13, 23\} | \ind(Q_{i}\tens Q_{j})_{L}=n \}|$ for $n=1, 2, 4$. Assume that $\ind(Q_{i})_{L}=2$ for all $i\in I_{3}$. If $G_1\geq 2$, then $(Q_{1})_{L}\simeq (Q_{2})_{L}\simeq (Q_{3})_{L}$, thus we take for $E$ a maximal subfield of $(Q_{1})_{L}$. If $G_{1}=1$, then $(Q_{i})_{L}\simeq (Q_{j})_{L}$ for some $i, j$, thus we take for $E$ the product of a maximal subfield of $(Q_{i})_{L}$ and a maximal subfield of the remaining quaternion. We now turn to the remaining case $G_1=0$: If $\ind(Q_{i}\tens Q_{j})_{L}=4$ and $\ind(Q_{1}\tens Q_{2}\tens Q_{3})_{L}\neq 8$, then we take for $E$ a maximal subfield of $(Q_{i}\tens Q_{j})_{L}$ which also splits the remaining quaternion. Similarly, if $\ind(Q_{i}\tens Q_{j})_{L}=2$, then we take for $E$ the product of a common maximal subfield of $(Q_{i})_{L}$ and $(Q_{j})_{L}$ and a maximal subfield of the remaining quaternion. Otherwise, we take for $E$ the product of maximal subfields of $(Q_{i})_{L}$ for all $i$. Hence,
\begin{equation*}
d:=[E:F]=
\begin{cases}
2 & \text{ if } G_1\geq 2,\\
4 & \text{ otherwise},\\
8 & \text{ if } \ind(Q_{1}\tens Q_{2}\tens Q_{3})_{L}=8.
\end{cases}
\end{equation*}

\begin{proposition}\label{threequadricsurface}
The torsion subgroup in the codimension $2$ Chow group of the product of three quadric surfaces with the same discriminant is contained in $(\Z/2\Z)^{\oplus 6}$.\end{proposition}
\begin{proof}
Let $Q_{i}$ be a quaternion $F$-algebra for $i\in I_{3}$ such that the corresponding qaudrics have the same discriminant. Let $X$ be the associated variety to the product of three quadric surfaces of $Q_{i}$ and $E$ be the splitting field of $X$ as above. 
If the discriminant is trivial, the result follows from Proposition \ref{threeconics}. Hence, we may assume that the discriminant is non-trivial, thus we have $X_{E}=R_{EL/E}(\P^{1})\times R_{EL/E}(\P^{1})\times R_{EL/E}(\P^{1})$ and $X_{EL}=(\P^{1})^{6}$. 

Let $x_{i}$ be the pullback of the class of tautological bundle on the projective line in $K(X_{EL})$ for $i\in I_{6}$. Set $J=\{\{1, 2\}, \{3, 4\}, \{5, 6\}\}$, $y_{i}=x_{i}-1$, and $y_{i_{1}\cdots i_{m}}=y_{i_{1}}\cdots y_{i_{m}}$ for $1\leq i_{1}<\cdots< i_{m}\leq 6$. It follows from the action of the Galois group $\Z/2\Z$ of $LE/E$ that we obtain the following bases of $K(X_{E})$ and $K(X)$, each of $36$ elements, respectively:
\[\{1, y_{p}+y_{q},\, y_{pq},\, y_{pr}+y_{qs},\, y_{rs}(y_{p}+y_{q}),\, y_{prt}+y_{qsu},\, y_{pqrs},\, y_{pq}(y_{rt}+y_{su}),\,
y_{pqrs}(y_{t}+y_{u}),
y_{123456}\},\]
where $p, q, r, s, t, u$ range over all numbers such that $\{\{p, q\}, \{r, s\}, \{t, u\}\}=J$, and
\begin{equation*}
\begin{split}
&\{1,\, e_{pq}(y_{p}+y_{q}),\, z_{pq},\, g_{tu}(z_{pr}+z_{qs}),\, e_{pq}z_{rs}(y_{p}+y_{q}),\, h(z_{prt}+z_{qsu}),\, z_{pq}z_{rs},\, g_{pq}z_{pq}(z_{rt}+z_{su}), \\
& e_{tu}z_{pq}z_{rs}(y_{t}+y_{u}), z_{12}z_{34}z_{56}\},
\end{split}
\end{equation*}
where $z_{pq}\!=y_{pq}+y_{p}+y_{q}, z_{pr}\!=y_{pr}+y_{p}+y_{r}, z_{prt}\!=y_{prt}+z_{pr}+y_{pt}+y_{rt}+y_{t}$, $e_{pq}=\ind(Q_{\max\{p,q\}/2})_{L}$, $g_{pq}=\ind(Q_{\max\{r,s\}/2}\tens Q_{\max\{t,u\}/2})_{L}$, and $h=\ind(Q_{1}\tens Q_{2}\tens Q_{3})_{L}$.

Note that any basis element of $K(X_{E})$ multiplied by $d$ is contained in the image of the restriction map and $|K(X_{E})/K(X)|=(e_{12}e_{34}e_{56}g_{12}g_{34}g_{56}h)^{4}$. Note also that if $h=1$, then
\begin{equation}\label{honeimagetwo}
y_{pr}+y_{qs}+y_{pt}+y_{qu}+y_{rt}+y_{su}-(y_{pq}+y_{rs}+y_{tu})=z_{prt}+z_{qsu}-(z_{pq}+z_{rs}+z_{tu})\in \Im(\res^{2/3})
\end{equation}
for all $\{\{p, q\}, \{r, s\}, \{t, u\}\}=J$, thus, by multiplying $y_{r}+y_{s}\in \Im(\res^{1/2})$ we get
\begin{equation}\label{honeimagethree}
(y_{r}+y_{s})(y_{pt}+y_{qu}-y_{pq}-y_{tu})+y_{rs}(y_{p}+y_{q}+y_{t}+y_{u})\in \Im(\res^{3/4}).
\end{equation}

Let $\alpha_{n}=|T^{n/n+1}(X_{E})/\Im(\res^{n/n+1})|$. We find upper bounds of $\alpha_{n}$ for $1\leq n\leq 6$. As $z_{pq}\in K(X)$, we get $y_{p}+y_{q}\in \Im(\res^{1/2})$, thus, $\alpha_{1}=1$. We divide the proof into two cases.


\framebox{{\it Case}: $G_{1}\geq 2$.} Then, $\ind(Q_{i})_{L}=d=h=2$ for all $i\in I_{3}$. If $g_{tu}=1$ for some $t, u$, then
\begin{equation}\label{threequadriccase1eq}
y_{pr}+y_{qs}-(y_{pq}+y_{rs})\!=z_{pr}+z_{qs}-(z_{pq}+z_{rs})\in T^{2}(X),
\end{equation}
thus, by multiplying $y_{p}+y_{q}, y_{t}+y_{u}\in \Im(\res^{1/2})$ we get
\begin{equation}\label{threequadriccase1eq2}
y_{pq}(y_{r}+y_{s})-y_{rs}(y_{p}+y_{q})\in \Im(\res^{3/4}), (y_{t}+y_{u})[y_{pq}(y_{r}+y_{s})-y_{rs}(y_{p}+y_{q})]\in \Im(\res^{4/5}).
\end{equation}
Hence, as $2y_{pq}, (y_{p}+y_{q})(y_{r}+y_{s})\in \Im(\res^{2/3})$ for all $\{p, q\}\neq\{r, s\}\!\in\! J$, we obtain $\alpha_{2}\leq 2^{3}$. As $(y_{p}+y_{q})(y_{r}+y_{s})(y_{t}+y_{u})\in \Im(\res^{3/4})$, it follows from (\ref{threequadriccase1eq2}) that $\alpha_{3}\leq 2^{6}$. Moreover, if $g_{tu}=1$ for some $t, u$ (i.e., $Q:=(Q_{\max\{p, q\}/2})_{L}\simeq (Q_{\max\{r, s\}/2})_{L}$), then it follows from the closed embedding $R_{L/F}(\SB(Q))\times R_{L/F}(\SB(Q_{\max\{t, u\}/2}))\hookrightarrow R_{L/F}(\SB(Q))\times R_{L/F}(\SB(Q))\times R_{L/F}(\SB(Q_{\max\{t, u\}/2}))$ that
\begin{equation}\label{threequadriccase1eq3}
y_{pqtu}+y_{rstu}+y_{tu}(y_{pr}+y_{qs})+y_{tu}(y_{ps}+y_{qr})\in \Im(\res^{4/5}).	
\end{equation}
Therefore, by (\ref{threequadriccase1eq2}) and (\ref{threequadriccase1eq3}) we have $\alpha_{4}\leq 2^{5}$. As $d=2$, we obtain $\alpha_{5}\leq 2^{3}$ and $\alpha_{6}\leq 2$, thus
$|\!\oplus T^{n/n+1}(X)_{\tors}|\leq 2^{2}$.

\framebox{{\it Case}: $G_{1}\leq 1$.} Let $g=g_{12}g_{23}g_{56}$. For simplicity, we shall assume that $\ind(Q_{i})_{L}=2$ for all $i\in I_{3}$, thus $|K(X_{E})/K(X)|=(8gh)^{4}$. Indeed, a simple modification of the proof below shows that the upper bounds of $\alpha_{2}\alpha_{3}\alpha_{4}\alpha_{5}\alpha_{6}$ for the cases $\ind(Q_{i})_{L}=1$ for some $i\in I_{3}$ are less than the maximum upper bound below.

It follows from the basis of $K(X)$ that $2y_{pq}\in T^{2}(X)$ and $(y_{p}+y_{q})(y_{r}+y_{s})\in \Im(\res^{2/3})$ for any $\{p, q\}\neq \{r, s\}\in J$. Moreover, we obtain $g_{tu}(y_{pr}+y_{qs})\in T^{2}(X)$ if $g_{tu}\neq 1$, thus $\alpha_{2}\leq 2^{3}g$ if $G_1=0$. If $g_{tu}=1$ for some $t, u$, then it follows from (\ref{threequadriccase1eq}) that $\alpha_{2}\leq 2^{3}g_{pq}g_{rs}=2^{3}g$. If $h=1$ (so $G_{2}=3$), then by using (\ref{honeimagetwo}) we can further reduce the upper bound $2^{6}$ to $2^{5}$, thus
\begin{equation*}\label{gequalonereferone}
\alpha_{2}\leq 2^{3}g \,\,(\text{resp.}\leq 2^{5}) \,\,\text{ if } h\neq 1 \,\,(\text{resp. otherwise}).
\end{equation*}
In codimension $3$, we have $2y_{rs}\cdot (y_{p}+y_{q}),\, (y_{1}+y_{2})(y_{3}+y_{4})(y_{5}+y_{6})\in \Im(\res^{3/4})$ for all $\{p, q\}\neq \{r, s\}\in J$. In addition, we get $g_{tu}(y_{pr}+y_{qs})\cdot (y_{t}+y_{u})\in \Im(\res^{3/4})$ if $g_{tu}\neq 1$. Hence, as $g_{tu}:=\max\{g_{12}, g_{34}, g_{56}\}=d=4$ except for the cases $h=1$ or $G_{2}=3$, we obtain $\alpha_{3}\leq 2^{6}\cdot d\cdot g_{pq}\cdot g_{rs}=2^{6}g$ (resp. $\leq 2^{5}\cdot d\cdot g_{pq}\cdot g_{rs}=2^{7}g$) if $G_1=0$ (resp. if $h=8$ \text{ or } $G_{2}=3$). Similarly, if $G_1=1$ (say, $g_{tu}=1$), then $d=4$ and $h=2$, thus it follows from (\ref{threequadriccase1eq2}) that $\alpha_{3}\leq 2^{5}\cdot d\cdot g_{pq}\cdot g_{rs}=2^{7}g$. Furthermore, if $h=1$, then by the same argument together with (\ref{honeimagethree}) we obtain $\alpha_{3}\leq 2^{6}\cdot d=2^{8}$.

In codimension $4$, we obtain $2y_{pq}\cdot 2y_{rs},\, (y_{p}+y_{q})\cdot  (y_{1}+y_{2})(y_{3}+y_{4})(y_{5}+y_{6})\in \Im(\res^{4/5})$ for all $\{p, q\}\neq \{r, s\}\in J$. Moreover, we get $g_{tu}(y_{pr}+y_{qs})\cdot 2y_{tu}\in T^{4}(X)$ if $g_{tu}\neq 1$. Therefore, $\alpha_{4}\leq 4^{3}\cdot (2\cdot 2g_{pq})(2\cdot 2g_{rs})(2\cdot 2g_{tu})=2^{12}g$ if $G_{1}=0$. Similarly, if $G_1=1$ (say, $g_{tu}=1$), then by multiplying $2y_{pq}\in T^{2}(X)$ to  (\ref{threequadriccase1eq}) we get $2y_{pqrs}\in T^{4}(X)$. Therefore, it follows from (\ref{threequadriccase1eq2}) and (\ref{threequadriccase1eq3}) that $\alpha_{4}\leq 2\cdot d^{6}=2^{13}$. In codimension $5$, as $2y_{pq}\in T^{2}(X)$ and $y_{p}+y_{q}\in \Im(\res^{1/2})$ for any $\{p, q\}\in J$, we get $\alpha_{5}\leq 2^{6}$. Obviously, $\alpha_{6}\leq d$.

Combining all bounds of $\alpha_{n}$, we see that $|\!\oplus T^{n/n+1}(X)_{\tors}|\leq 2^{6}$ except the cases where $h=1$ or $G_1=1, G_2=2$. In these cases, we can further reduce the upper bounds $2^{12}$ and $2^{10}$, respectively as follows. If $g_{pq}=2$, then $2y_{rs}(y_{tu}+y_{t}+y_{u})-g_{pq}(y_{rt}+y_{su})(y_{rs}+y_{r}+y_{s})=2y_{rstu}\in T^{3}(X)$ and $2y_{rstu}(y_{pq}+y_{p}+y_{q})\in T^{4}(X)$. Hence, if $2y_{rstu}\not\in T^{4}(X)$, then the class of $2y_{rstu}$ gives a torsion of order $2$ in $T^{3/4}(X)$. Similarly, if $2y_{rstu}(y_{pq}+y_{p}+y_{q})\not\in T^{5}(X)$, the its class gives a torsion of order $2$ in $T^{4/5}(X)$. Therefore, if $h=1$ (so $G_{2}=3$) or $G_1=1, G_2=2$, we have $|T^{2/3}(X)_{\tors}|\leq 2^{6}$ in both cases.\end{proof}

\paragraph{\bf Acknowledgments.} 
The author would like to thank N.~Karpenko for helpful discussion. This work was partially supported by an internal fund from KAIST, TJ Park Junior Faculty Fellowship of POSCO TJ Park Foundation, and National Research Foundation of Korea (NRF) funded by the Ministry of Science, ICT and Future Planning (2013R1A1A1010171).

\end{document}